%% file: anisotropic-nonlipschitz.tex
\documentclass{cmslatex}
\usepackage[paperwidth=7in, paperheight=10in, margin=.875in]{geometry}
 \usepackage[backref,colorlinks,linkcolor=red,anchorcolor=green,citecolor=blue]{hyperref}
\usepackage{amsfonts,amssymb}
\usepackage{amsmath}
\usepackage{graphicx}
\usepackage{cite}
\usepackage{enumerate}
\renewcommand{\theequation}{\arabic{section}.\arabic{equation}}

\usepackage{mathrsfs,subfigure,epstopdf,float,multirow,algorithmic,algorithm,cases,color}

\newcommand{\rS}{\mathrm{S}}
\newcommand{\rJ}{\mathrm{J}}
\newcommand{\rT}{\mathrm{T}}
\newcommand{\rI}{\mathrm{I}}
\newcommand{\RR}{\mathbb{R}}

   \allowdisplaybreaks
\begin{document}
 \title{On anisotropic non-Lipschitz restoration model: lower bound theory and convergent algorithm\thanks{Received date, and accepted date (The correct dates will be entered by the editor).}}


          \author{Chunlin Wu\thanks{School of Mathematical Sciences, Nankai University, Tianjin 300071,
          		China (wucl@nankai.edu.cn, 1120200026@mail.nankai.edu.cn).}
          \and Xuan Lin\footnotemark[2]
      \and Yufei Zhao\footnotemark[2]\thanks{Corresponding author. School of Mathematical Sciences, Nankai University, Tianjin 300071,	China (matzyf@nankai.edu.cn).}
  }

         \pagestyle{myheadings}
         \markboth{ANISOTROPIC NON-LIPSCHITZ MODEL: LOWER BOUND THEORY AND ALGORITHM}{CHUNLIN WU, XUAN LIN, YUFEI ZHAO} 
         \maketitle

          \begin{abstract}
       For nonconvex and nonsmooth restoration models, the lower bound theory reveals their good edge recovery ability, and related analysis can help to design convergent algorithms.	Existing such discussions are focused on isotropic regularization models, or only the lower bound theory of anisotropic model with a quadratic fidelity. 
       In this paper, we consider a general image recovery model with a non-Lipschitz anisotropic composite regularization term and an $\ell_q$ norm ($1\leq q<+\infty$) data fidelity term.
       We establish the lower bound theory for the anisotropic model with an $\ell_1$ fidelity, which applies to impulsive noise removal problems.
       For the general case with
       $1\leq q<+\infty$, a support inclusion analysis is provided. To solve this non-Lipschitz composite minimization model,
       we are then motivated to introduce a support shrinking strategy in the iterative algorithm and relax the support constraint to a thresholding support constraint, which is more computationally practical.
       The objective function at each iteration is also linearized to construct a strongly convex subproblem. To make the algorithm more implementable,
       we compute an approximation solution to this subproblem at each iteration, but not exactly solve it.
       The global convergence result of the proposed inexact iterative thresholding and support shrinking algorithm with proximal linearization is established. The experiments on image restoration and two stage image segmentation demonstrate the effectiveness of the proposed algorithm.
          \end{abstract}
\begin{keywords}  image restoration; anisotropic model; non-Lipschitz optimization; lower bound theory;  thresholding; support shrinking
\end{keywords}

 \begin{AMS} 49K30; 49N45; 90C26; 94A08
\end{AMS}
          \section{Introduction}\label{intro}
          \input{section1}

          \section{Lower bound theory and support inclusion analysis}\label{sec:lowerbound}

\input{section2}

          \section{Algorithms and convergence analysis}\label{sec:alg}

\input{section3}

          \section{Algorithm implementation}\label{sec:alg-detail}

          \input{section4}

          \section{Experiments}\label{sec:experiment}

\input{section5}

          \section{Conclusions}\label{sec:con}
          
          \input{section6}

          \appendix
          \section*

\input{section7}

\bibliographystyle{elsarticle-num}
\bibliography{lowerbound_ref}

          \end{document}

%% file: section1.tex
		We consider the problem of image recovery, i.e., trying to find the unknown ground truth $\underline{x}\in\RR^N$ from the degraded observation $b\in\RR^M$ as follows
	$$
	b=A\underline{x}+n.
	$$
	Herein, $M$ and $N$ are positive integers, $n\in\RR^M$ represents the noise,
	and the matrix $A\in\RR^{M\times N}$ is related to the information acquisition process, e.g., an identity matrix for the image denoising problem, a convolution matrix generated by blur kernel for image blurring, and a measurement matrix in the problem of compressed sensing. 
	This linear inverse problem is usually ill posed.
	To solve such a problem, we consider the following 
	minimization model
	
	\begin{equation}\label{eqn:lqlp1}
	\min_{x\in\RR^N} \sum_{i\in \rJ}\phi(|G_i^\top x|)+\frac{\beta}{q}\|Ax-b\|_q^q,
	\end{equation}
	where  $1\leq q<\infty$, $\phi$ is a potential function satisfying certain properties,
	$\rJ$ is an index set, $\{ G_i\}_{i\in \rJ}\subset\RR^N$ are the columns of matrix $G\in\RR^{N\times\sharp\rJ}$ which can be a set of sparsifying operators
	(e.g., the discrete gradient operators or the atoms in a wavelet transform),
	and $G_i^\top x$ denotes the difference of each pixel with its neighbors.
	As one can see, the $\ell_q$ norm $\|\cdot\|_q$ $(1\leq q<+\infty)$ is adopted in the data fidelity term of model \eqref{eqn:lqlp1}
	and
	it can be used to handle different additive noises by choosing different values of $q$ from a maximum \textit{a posterior} (MAP) derivation.
	For example, people use the squared $\ell_2$ norm  fidelity
	to deal with the white Gaussian noise (\cite{Bouman1993A,Chen2001Atomic,candes2006stable}), and the $\ell_1$ norm fidelity for
	 impulse noise removal (\cite{nikolova2004a,granai2005,rodriguez2008,Fuchs2009Fast}).
	Meanwhile, the 
	energy of the discrete derivatives of $x$ is penalized in the regularization term and the potential functions $\phi$ considered in this paper are nonconvex and non-Lipschitz.
	
	Nonconvex and nonsmooth regularizers have the advantages in finding sparse solutions than convex ones and they can help to preserve or generate neat edges in restored images (\cite{nikolova2005analysis,chartrand2007exact,Foucart2009Sparsest}).
	The very useful edge property of nonconvex and nonsmooth regularization in image restoration is due to the lower bound theory, which is to provide a uniform lower bound of nonzero image differences of local minimizers.
	One question we are interested in about model \eqref{eqn:lqlp1} is whether it satisfies the lower bound theory.
The other crucial question we are concerned about is how to design efficient and convergent algorithms for solving this model.
	Compared to the convex minimization model,
	the difficulty for solving model \eqref{eqn:lqlp1} arises mainly from the possible nonsmoothness of the fidelity term, the nonconvex and nonsmooth property of the regularization term, as well as the composition of $\phi$ and $G_i$'s, which makes the model a composite optimization problem.
	These properties make it a non-trivial task for solving the minimization model \eqref{eqn:lqlp1}.
	In the following two paragraphs, we review existing works studying these two problems.

	The lower bound theory is theoretical and has been studied in
	the literatures; see, e.g., \cite{nikolova2005analysis,nikolova2008Efficient,nikolova2010fast,chen2012non,2018fengOn,zeng2019nonlip,liu2019a,2019zengOn}.
	In general, the key to derive the lower bound theory is to apply the inequality derived from the optimality condition to some carefully constructed testing vectors.
For the model with an anisotropic regularizer as given in \eqref{eqn:lqlp1}, only the case of a squared $\ell_2$ norm fidelity term has been considered; see the pioneering work \cite{nikolova2005analysis}, or \cite{chen2012non} where box constraints were added to the minimization model.
	The technique used in \cite{nikolova2005analysis} and \cite{chen2012non} is based on the second order optimality condition, which cannot be applied to model \eqref{eqn:lqlp1} with other fidelity terms.
	Recently, the authors of \cite{zeng2018edge} and \cite{zeng2019nonlip} discussed the lower bound theory of models with non-Lipschitz isotropic regularization and different fidelities. This theory for nonconvex and nonsmooth isotropic models with box constraints was established in \cite{2019zengOn} and \cite{2018fengOn}, the latter of which focuses on the case of $\ell_0$ "norm" potential function. The derivation in \cite{zeng2018edge,zeng2019nonlip,2019zengOn,2018fengOn}  is based on the conservation of image gradient fields or the image-gradient relationship, and at the same time the discussion is restricted to the case that the regularization quantity corresponding to each pixel is only one.
	These limitations make such techniques difficult to be generalized to a regularization with an arbitrary sparsifying system.
			Also, the existing construction of testing vectors is either only suitable for the second order optimality condition, or restricted to the case of the isotropic first order regularization.
	Indeed, the lower bound theory of the anisotropic model \eqref{eqn:lqlp1} with non-quadratic fidelity terms has not been studied in existing works.

	For nonconvex and nonsmooth composite minimization models, the other crucial problem is how to design effective and convergent algorithms to solve them.
	One straightforward approach is to utilize variable splitting and alternating direction method of multipliers (ADMM) \cite{glowinski1989augmented} or equivalently, split Bregman iteration \cite{goldstein2009the}, but
	their convergence for nonconvex optimization problems is only guaranteed under specific assumptions like surjectiveness on the linear operator \cite{li2015global,hong2016convergence,wang2019global}, which do not hold in 2D image restoration models with gradient or higher order sparsifying systems.
	In the existing studies, one popular class of approaches are smoothing approximation methods, e.g. \cite{chen2012smoothing,bian2015linearly,chen2013optimality,hintermuller2014a,chan2014half}, in which smoothing functions with auxiliary parameters are utilized to approximate the original objective function.
	These include the smoothing gradient method (\cite{chen2012smoothing}), the smoothing quadratic regularization (SQR) method (\cite{bian2015linearly}), the smoothing trust region Newton method (\cite{chen2013optimality}) and the $R$-regularized Newton scheme (\cite{hintermuller2014a}) both for twice differentiable
	fidelity terms, the half-quadratic technique based method (\cite{chan2014half}) for specific potential function $\phi(t)=t^p$ with $0<p,q\leq2$.
In such methods, it was usually only proved that the iterative sequence converges to a stationary point of the smoothed variational model, or satisfies the subsequence convergence result by letting the smoothing parameter tend to zero.
Another class of approaches is the iteratively reweighted methods, which concentrate on the special case of $\phi(t)=t^p$.
Under the assumption that $0<p,q\leq 2$, the iterative reweighted norm (IRN) algorithm (\cite{rodriguez2009efficient}) and the generalized Krylov subspace method for $\ell_p$-$\ell_q$ minimization (GKSpq, \cite{lanza2015krylov}) are proposed,
where the $\ell_p$ (or $\ell_q$) term is approximated by a weighted $\ell_2$ norm with iteratively updated weighting matrices.
However, the convergence analysis of iterate sequences in \cite{rodriguez2009efficient} and \cite{lanza2015krylov} 
only focuses on the case of $1\leq p,q\leq 2$, which is not applicable to the more interesting case with nonconvex regularization terms.
The last approach we mention here is a support shrinking strategy derived respectively for various signal recovery problems \cite{xue2019efficient,liunew,feng2020the} and different image restoration models with isotropic first order regularization \cite{zeng2019iterative,zeng2019nonlip,zheng2020a,Gao2021a}, yielding some (two-loop) iterative algorithms with either exact or inexact inner loops, when incorporated with the iteratively reweighted $\ell_1$ \cite{Candes2008IRL1,Lai2009IRL1} or least squares \cite{Daubechies2010IRLS,Lai2013IRLS}.
During a current study on nonconvex wavelet image restoration, we got aware that a similar support shrinking strategy was also induced from a majorization-minimization framework for non-Lipschitz synthesis wavelet model \cite{Figueiredo2007Majorization}, but without convergence result.
In such works with convergence results,
the proof techniques for non-Lipschitz composite optimization problems all use the conservation of image gradient fields, which cannot be used to composite models with the anisotropic regularization by a general sparsifying system.
To our best knowledge, existing algorithms for anisotropic composite minimization models have no proved global sequence convergence to a stationary point of the original objective function.

As can be seen, the research on the lower bound theory and related algorithm study for nonconvex anisotropic regularization models is very limited.
Also, existing related derivation techniques cannot be applied to analyze anisotropic regularization models with non-quadratic fidelities, or iterative algorithms solving a general anisotropic model.
In this paper, we aim to analyze and solve the general non-Lipschitz anisotropic regularization model \eqref{eqn:lqlp1}.
	The key technique we introduce is a new construction of the testing vectors,
	based on which, the first order optimality condition can be applied to prove a lower bound theory and to analyze the algorithm in the situation of anisotropic regularization.
	The main contributions of this paper can be summarized as follows.
	\begin{itemize}
		\item[(1)]
		We show a lower bound theory for the composite minimization model \eqref{eqn:lqlp1} with an $\ell_1$ data fitting term, which is useful for impulse noise removal.
		For the general case with 
		$1\leq q<\infty$, a support inclusion analysis is provided.
		Such theoretical results apply to anisotropic  regularization by a general sparsifying system. They help us not only understand more about the model, but also to design iterative algorithms.
		\item[(2)] We propose an inexact version of an iterative thresholding and support shrinking algorithm with proximal linearization to solve model \eqref{eqn:lqlp1}. At each iteration, the algorithm thresholds the differences to determine the support, and construct a strongly convex problem, which can be easily solved.
This makes the algorithm more practical and computationally efficient. The sequence convergence of the generated iterates is also
		established.
	\end{itemize}

	The remainder of this paper is organized as follows.
	In Section \ref{sec:lowerbound}, we investigate the lower bound theory and support inclusion property of the non-Lipschitz anisotropic model \eqref{eqn:lqlp1}.
	Based on these conclusions, in Section \ref{sec:alg}, we introduce an iterative algorithm with the strategy of thresholding and support shrinking.
	Then we consider an inexact version of the algorithm and establish its convergence result. We present the details of algorithm implementation in Section \ref{sec:alg-detail}.
	In Section \ref{sec:experiment}, the experiments on image deconvolution and two stage image segmentation are carried out to exhibit the usefulness of the proposed algorithm.
	We conclude this paper in Section \ref{sec:con}.

	We now give some notations, which will be used throughout this paper.
	For a set $\rS$, let $\sharp \rS$ or $|\rS|$ denote its cardinality.
	For a matrix $B$, we use $B_i$ to denote the $i$'th column of $B$,  use $B^\top$ and $B_i^\top$ to denote the transposes of $B$ and $B_i$, respectively.  For a vector $x\in\RR^N$, $x_l$ is the $l$'th entry of $x$, $1\leq l\leq N$.
	Given $x\in\RR^N$, the $\ell_p$ (quasi-)norm $\|x\|_p$ with $p\in(0,+\infty)$ is defined by $\|x\|_p=\left(\sum_{1\leq l\leq N}|x_l|^p\right)^{1/p}$.
	We denote the set of indices corresponding to nonzero sparsifying coefficients $G_i^\top x$ of $x$ as the coefficient support (or, abbreviated as support) $\rS(x)$, i.e., $\rS(x)=\{ i\in \rJ:G_i^\top x\neq0 \}$.
The set of indices corresponding to differences $G_i^\top x$ whose absolute values are larger than $\tau$ ($\tau\geq0$), is called the coefficient $\tau$-support (or, abbreviated as $\tau$-support) of $x$ and written as
$\rT(x)=\{ i\in \rJ:|G_i^\top x|>\tau \}$.
Specifically, if taking $\tau=0$, $\rT(x)$ degenerates to $\rS(x)$.

%% file: section2.tex
	For the convenience of description, we denote
	\begin{equation}\label{eqn:oflqlp}
	\mathcal{F}(x):=\sum_{i\in \rJ}\phi(|G_i^\top x|)+\frac{\beta}{q}\|Ax-b\|_q^q,
	\end{equation}
	for our considered model \eqref{eqn:lqlp1}, and rewrite it as
    \begin{equation}\label{eqn:lqlp}
	(\mathfrak{F})\quad\min_{x\in\RR^N} \mathcal{F}(x).
	\end{equation}

	The potential function $\phi$ is supposed to satisfy the general assumptions.
	\begin{assumption}\label{thm:xp}
		\begin{itemize}
			\item[(a)] $\phi:[0,+\infty)\to[0,+\infty)$ is a continuous concave coercive function, 
			and $\phi(0)=0$.
			\item[(b)] $\phi$ is $C^1$ on $(0,+\infty)$, $\phi^\prime(t)|_{(0,+\infty)}>0$ and $\phi^\prime(0+)=+\infty$.
			\item[(c)] For any $\alpha>0$, $\phi^\prime$ is $L_\alpha$-Lipschitz continuous on $[\alpha,+\infty)$, i.e., there exists a constant $L_\alpha$ determined by $\alpha$, such that for any $t,s\in[\alpha,+\infty)$, $|\phi^\prime(t)-\phi^\prime(s)|\leq L_\alpha|t-s|$.
		\end{itemize}
	\end{assumption}

\begin{example}
	Two examples of potential functions satisfying Assumption \ref{thm:xp} are $\phi_1(t)=t^p$ ($0<p<1$) and $\phi_2(t)=\log(1+t^p)$ $(0<p<1)$ (\cite{zeng2019iterative,zheng2020a}).
	
	\end{example}

	By Assumption~\ref{thm:xp}, 
	the subdifferential of $\phi(|t|)$ at $t$ is given by
	$$
	\partial\phi\circ|\cdot|(t)=
	\begin{cases}
	(-\infty,+\infty),&\quad\text{if }t=0,\\
	\frac{t}{|t|}\phi^\prime(|t|) ,&\text{otherwise},
	\end{cases}
	$$
	and $\phi(|t|)$ is subdifferentially regular (\cite{rockafellar2009variational}) at any $t\in\RR^N$.

	Throughout this paper, $A$ and $G$ are assumed to satisfy the following basic property, which is trivial in image restoration problems \cite{Wang2008,Wu2010Augmented}.

	\begin{assumption}\label{assumption:ag}
	$\ker A\cap\ker G^\top=\{ 0\}$.	
	\end{assumption}

	Following the arguments in existing works, e.g., \cite{zheng2020a} and \cite{lou2015weighted}, one can show the coercive property of the objective function $\mathcal{F}(x)$.

	\begin{theorem}\label{thm:coercive}
		Suppose that Assumption \ref{thm:xp} and Assumption \ref{assumption:ag} hold true. Then the function $\mathcal{F}(x)$ in \eqref{eqn:oflqlp} is coercive and thus \eqref{eqn:lqlp} has at least one solution.
	\end{theorem}

In this section, we focus on discussing the lower bound theory and related support inclusion analysis for the
model \eqref{eqn:lqlp1}, i.e., the minimization problem \eqref{eqn:lqlp}.
In \cite{nikolova2005analysis}, the authors studied the lower bound theory for a special case of \eqref{eqn:lqlp} with $q=2$, i.e.,
	\begin{equation}\label{thm:lpl2}
\min_{x\in\RR^N} \sum_{i\in \rJ}\phi(|G_i^\top x|)+\frac{\beta}{2}\|Ax-b\|_2^2.
\end{equation}
To establish such a conclusion, the second order derivative of $\phi$ needs to be considered.
	We revise the conclusion in \cite{nikolova2005analysis} to Theorem \ref{thm:lbl2} below, in which the requirement on $\phi$ is  modified to be consistent with Assumption \ref{thm:xp}.
\begin{theorem}\label{thm:lbl2}
Let $\phi$ be the potential function satisfying Assumption~\ref{thm:xp} and the followings:
	$\phi$ is $C^2$ on $(0,+\infty)$,
	$\phi''$ is increasing on $(0,+\infty)$ with $\phi''(t)\leq0$ for any $t>0$, $\phi''(0^+)=-\infty$ and $\lim\limits_{t\to+\infty}\phi''(t)=0$.
Then there exists a constant 
$\theta_2>0$, such that
for any local minimizer $x^\ast$  of model \eqref{thm:lpl2} (i.e., the problem \eqref{eqn:lqlp} with $q=2$),
the difference $G_i^\top x^\ast$ is either zero or satisfies $|G_i^\top x^\ast|\geq \theta_2$. 
\end{theorem}

\begin{proof}
	The proof is similar to that in \cite{nikolova2005analysis}.
\end{proof}

Theorem \ref{thm:lbl2} is for the case of Gaussian measurement noise. Another common and important situation is the impulse noise removal, and the corresponding model is as follows,
	\begin{equation}\label{eqn:model1}
	\min_{x\in\RR^N} \sum_{i\in \rJ}\phi(|G_i^\top x|)+\beta\|Ax-b\|_1,
	\end{equation}
	which is the $q=1$ case of \eqref{eqn:lqlp}. The key technique for proving Theorem~\ref{thm:lbl2} is the second order necessary condition on local minimizers with the corresponding testing vector construction \cite{nikolova2005analysis}, which cannot be applied to deal with the model \eqref{eqn:model1}.
	For the model \eqref{eqn:model1}, we will instead establish a lower bound theory on the stationary points (including local minimizers) by using the first order optimality  condition with a new construction of testing vector.

\begin{theorem}\label{thm:local-min-lowerbd}
There exists a constant $\theta_1>0$, such that for any stationary point $x^\ast$ of \eqref{eqn:model1} (i.e., the problem \eqref{eqn:lqlp} with $q=1$), it satisfies
$$
\text{either }G_i^\top  {x}^\ast=0\text{ or }|G_i^\top  {x}^\ast|\geq\theta_1,\quad\forall i\in\rJ.
$$
\end{theorem}

\begin{proof}
Suppose that ${x}^\ast$ is a stationary point of \eqref{eqn:model1}. Define
$$
{\rS}^\ast=\rS(x^\ast)=\left\{ i\in \rJ: G_i^\top  {x}^\ast\neq0 \right\} \text{ and }
K(\rS^\ast)=\left\{ x\in\RR^N:\  G_i^\top x=0\ \forall i\in \rJ\backslash \rS^\ast \right\}.
$$

Since ${x}^\ast$ is a stationary point of $\mathcal{F}(x)$,
$s^\ast=0\in\partial \mathcal{F}({x}^\ast)$ and
\begin{equation}\label{eqn:first-order}
(s^\ast)^\top v=0\quad\forall v\in K({\rS}^\ast).
\end{equation}

Since all the conditions in Corollary 10.9 of \cite{rockafellar2009variational} hold, we derive
		\begin{align*}
\partial \mathcal{F}(x^\ast)
&= \partial\left(\beta \|Ax^\ast-b\|_1 \right) + \sum_{i\in\rS^\ast} \phi^\prime(|G_i^\top x^\ast|)\frac{G_i^\top x^\ast}{|G_i^\top x^\ast|}G_i + \partial\left( \sum_{i\in\rJ\backslash\rS^\ast}\phi(|G_i^\top x^\ast|) \right) \\
(\text{by \cite[Theorem 2.3]{zheng2020a} }) &=  \beta\partial \|\cdot\|_1 (Ax^\ast-b) + \sum_{i\in{\rS}^\ast} \phi^\prime(|G_i^\top x^\ast|)\frac{G_i^\top x^\ast}{|G_i^\top x^\ast|}G_i+\sum_{i\in\rJ\backslash{\rS}^\ast}(\ker G_i^\top)^\perp\\
& \subset
\beta A^\top [-1,1]^M
+\sum_{i\in{\rS}^\ast} \phi^\prime(|G_i^\top x^\ast|)\frac{G_i^\top x^\ast}{|G_i^\top x^\ast|}G_i+\sum_{i\in\rJ\backslash{\rS}^\ast}(\ker G_i^\top)^\perp.
\end{align*}

By \eqref{eqn:first-order}, there exists $d\in[-1,1]^M$, such that for any $v\in K(\rS^\ast)$,
$$
\sum_{i\in {\rS}^\ast}\phi^\prime(|G_i^\top {x}^\ast|)\frac{G_i^\top  {x}^\ast}{|G_i^\top  {x}^\ast|}G_i
^\top v
=-\beta d^\top Av
\leq
\beta\sum_{l=1}^M\|(A^\top)_l\|_2\|v\|_2=\alpha\|v\|_2,
$$
where $\alpha=\beta\sum\limits_{l=1}^M\|(A^\top)_l\|_2$ and $(A^\top)_l$ denotes the $l$'th column of $A^\top$.

Define
$$
{\rS}_+^\ast=\left\{ i\in{\rS}^\ast:G_i^\top {x}^\ast>0 \right\}\subset{\rS}^\ast.
$$
We consider a fixed $j\in{\rS}^\ast$,
and define
$$
V({\rS}^\ast,j,{\rS}_+^\ast)
=\left\{ v\in K(\rS^\ast): G_i^\top v>0\ \forall i\in \rS_+^\ast,\quad G_i^\top v<0\ \forall i\in \rS^\ast\backslash \rS_+^\ast,\quad |G_j^\top v|=1 \right\}
$$
which is nonempty, since
$v=\frac{{x}^\ast}{|G_j^\top {x}^\ast|}\in V({\rS}^\ast,j,{\rS}_+^\ast)$.
Let $\tilde{v}=\tilde{v}({\rS}^\ast,j,{\rS}_+^\ast)$ be the solution to the following problem
$$
\min\|v\|_2\quad\text{subject to }v\in \overline{V({\rS}^\ast,j,{\rS}_+^\ast)},
$$
where $\overline{V({\rS}^\ast,j,{\rS}_+^\ast)}$ denotes the closure of $V({\rS}^\ast,j,{\rS}_+^\ast)$. Then it implies that
$$
\tilde{v}\in K({\rS}^\ast);\text{ and for any } i\in{\rS}^\ast,\text{ either }G_i^\top {x}^\ast\cdot G_i^\top \tilde{v}>0\text{ or }G_i^\top \tilde{v}=0;\text{ and } |G_j^\top \tilde{v}|=1.
$$

	Similar to the definition of $K(\rS^\ast)$ and  $V({\rS}^\ast,j,{\rS}_+^\ast)$, for any set $\rI\subset \rJ$,   define
	$$
	K(\rI)=\left\{ u\in\RR^N:\  G_i^\top u=0\ \forall i\in \rJ\backslash \rI \right\},
	$$
	and for any subset $\rI_{+}\subset \rI$ and index $j\in\rI$, define
	$$
	V(\rI,j,\rI_+)=\{ v\in K(\rI): G_i^\top v>0\ \forall i\in \rI_+,\quad G_i^\top v<0\ \forall i\in \rI\backslash \rI_+,\quad |G_j^\top v|=1 \}.
	$$
	For those nonempty $V(\rI,j,\rI_+)$,
	let $\tilde{v}(\rI,j,\rI_+)$ be the solution to the following constrained optimization problem
	$$
	\min\|v\|_2\quad\text{ subject to }v\in \overline{V(\rI,j,\rI_+}).
	$$
Denote
$$
\mu(\rI)=\max_{\substack{j\in \rI, \rI_+\subset \rI
}}   \|\tilde{v}(\rI,j,\rI_+)\|_2
$$
and
\begin{equation}\label{eqn:mu}
\mu=\max_{\rI\subset\rJ }\{ \mu(\rI)\},
\end{equation}
where $\mu$ is well-defined and positive.
Then $\mu$ is a uniform upper bound of $\|\tilde{v}\|_2$, i.e.,
$\|\tilde{v}\|_2=\|\tilde{v}({\rS}^\ast,j,{\rS}_+^\ast)\|_2\leq\mu$ for any stationary point $x^\ast$.

By the assumption on function $\phi(\cdot)$ and the fact that for any $i\in{\rS}^\ast$, either $G_i^\top {x}^\ast\cdot G_i^\top \tilde{v}>0$ or $G_i^\top \tilde{v}=0$, we have
$$
\phi^\prime(|G_i^\top {x}^\ast|)\frac{G_i^\top {x}^\ast}{|G_i^\top {x}^\ast|}G_i
^\top \tilde{v}\geq0\quad\forall i\in{\rS}^\ast.
$$
Therefore,
$$
\phi^\prime(|G_j^\top {x}^\ast|)=
\phi^\prime(|G_j^\top {x}^\ast|)\frac{G_j^\top {x}^\ast}{|G_j^\top {x}^\ast|}G_j
^\top \tilde{v}
\leq\sum_{i\in{S}^\ast}\phi^\prime(|G_i^\top {x}^\ast|)\frac{G_i^\top {x}^\ast}{|G_i^\top {x}^\ast|}G_i
^\top \tilde{v}\leq\alpha\|\tilde{v}\|_2\leq\alpha\mu.
$$

	Since $\phi^\prime(0+)=+\infty$, the constant $\theta_1=\inf\{ t>0:\phi^\prime(t)\leq\alpha\mu \}>0$ is well defined and independent of ${x}^\ast$.
	Thus, we obtain that
	$$
	|G_j^\top {x}^\ast|\geq\theta_1,
	$$
	which holds for any $j\in {\rS}^\ast$ and any stationary point ${x}^\ast$.
	\end{proof}
\begin{remark}
	By Theorem \ref{thm:lbl2}, the lower bound theory holds for the local minimizers of the objective function $\mathcal{F}(x)$ in \eqref{eqn:oflqlp}, in the case of squared $\ell_2$ norm fidelity, i.e., $q=2$.
	By Theorem \ref{thm:local-min-lowerbd}, in the case of $\ell_1$ norm fidelity term, i.e., $q=1$, the lower bound theory
	can hold for all stationary points (including the local minimizers) of $\mathcal{F}(x)$.
\end{remark}
	
	By Theorem \ref{thm:lbl2} and
Theorem~\ref{thm:local-min-lowerbd}, there exists some constant $\theta$, such that for any local minimizer $x^\ast$ of $\mathcal{F}(x)$ with $q=2$ (or any stationary point $x^\ast$ of $\mathcal{F}(x)$ with $q=1$), it satisfies
$$
\text{either }G_i^\top  {x}^\ast=0\text{ or }|G_i^\top  {x}^\ast|\geq\theta.
$$
Suppose that  ${x}^\ast$ is very near to a given point $\tilde{x}$, such that $|G_i^\top {x}^\ast-G_i^\top \tilde{x}|\leq\|G_i\|_2\|{x}^\ast-\tilde{x}\|_2<\theta,\ \forall i\in \rJ$. Then $|G_i^\top {x}^\ast|\leq |G_i^\top {x}^\ast-G_i^\top \tilde{x}| + |G_i^\top \tilde{x}| <\theta+|G_i^\top \tilde{x}|$, which implies \textit{the support inclusion property}, i.e., for
$i\in\rJ$,
$$
G_i^\top {x}^\ast=0,\quad\text{when }G_i^\top \tilde{x}=0.
$$

In fact, although there are so far no lower bound theory for model \eqref{eqn:lqlp} with $q\in[1,+\infty)$ and $q\neq 1,2$,
the support inclusion property
can be shown to hold for the stationary points of model \eqref{eqn:lqlp} with any $q\in[1,+\infty)$ 
as stated in Theorem~\ref{thm:motivation}.

\begin{theorem}\label{thm:motivation}
Given $\tilde{x}\in\RR^N$, assume that a stationary point $\hat{x}$ of \eqref{eqn:lqlp} is sufficiently close to $\tilde{x}$. Then for $i\in \rJ$,
$$
G_i^\top \hat{x}=0,\quad\text{when  }G_i^\top \tilde{x}=0.
$$
\end{theorem}
\begin{proof}
	At first, we assume  that there exists a  stationary point $\hat{x}$ of \eqref{eqn:lqlp}  satisfying $\|\hat{x}-\tilde{x}\|<\|\tilde{x}\|$, 
	which implies $\|\hat{x}\|<2\|\tilde{x}\|$.
	Suppose that
	for some $j\in \rJ$,  $G_j^\top \tilde{x}=0$ but $G_j^\top \hat{x}\neq 0$.
	Define
	$$
	\hat{\rS}=\rS(\hat{x})=\{ i\in \rJ:G_i^\top\hat{x}\neq0 \}\text{ and }
	K(\hat{\rS})=\left\{ x\in\RR^N:\  G_i^\top x=0\ \forall i\in \rJ\backslash \hat{\rS} \right\}.
	$$
	Since $\hat{x}$ is a stationary point of \eqref{eqn:lqlp},
	$\hat{s}=0\in\partial\mathcal{F}(\hat{x})$ and
	\begin{equation}\label{eqn:lqlp-first-order}
	\hat{s}^\top v=0\quad\text{for any } v\in K(\hat{\rS}).
	\end{equation}
	Since all the conditions in Corollary 10.9 of \cite{rockafellar2009variational} hold, we derive
		\begin{align*}
	\partial \mathcal{F}(\hat{x})
	&= \partial\left(\frac{\beta}{q} \|A\hat{x}-b\|_q^q \right) + \sum_{i\in\hat{\rS}} \phi^\prime(|G_i^\top \hat{x}|)\frac{G_i^\top \hat{x}}{|G_i^\top \hat{x}|}G_i + \partial\left( \sum_{i\in\rJ\backslash\hat{\rS}}\phi(|G_i^\top \hat{x}|) \right) \\
	&=  \frac{\beta}{q}A^\top\partial \|\cdot\|_q^q (A\hat{x}-b) + \sum_{i\in\hat{\rS}} \phi^\prime(|G_i^\top \hat{x}|)\frac{G_i^\top \hat{x}}{|G_i^\top \hat{x}|}G_i+\sum_{i\in\rJ\backslash\hat{\rS}}(\ker G_i^\top)^\perp.
	\end{align*}
	By \eqref{eqn:lqlp-first-order}, for any $v\in K(\hat{\rS})$,
	$$
	\sum_{i\in \hat{\rS}}\phi^\prime(|G_i^\top \hat{x}|)\frac{G_i^\top \hat{x}}{|G_i^\top \hat{x}|}G_i
	^\top v
	\leq
	\begin{cases}
	{\beta}\sum_{l=1}^M \left(\|(A^\top)_l\|_2\|\hat{x}\|_2+\|b_l\|_2 \right)^{q-1}\|(A^\top)_l\|_2\|v\|_2, & q>1 \\
	\beta\sum_{l=1}^M\|(A^\top)_l\|_2\|v\|_2, & q=1
	\end{cases}
	$$
Since $\|\hat{x}\|_2<2\|\tilde{x}\|_2$, we have	
	$$
	\sum_{i\in \hat{\rS}}\phi^\prime(|G_i^\top \hat{x}|)\frac{G_i^\top \hat{x}}{|G_i^\top \hat{x}|}G_i
	^\top v
	<
	\alpha_1(\tilde{x})\|v\|_2,
	$$
	where $\alpha_1(\tilde{x})
	=
	\beta\sum\limits_{l=1}^M \|(A^\top)_l\|_2
	\left(\max\left\{ 2\|(A^\top)_l\|_2\|\tilde{x}\|_2+\|b_l\|_2 ,1
	\right\}\right)^{q-1}$.

	By using similar notations in Theorem \ref{thm:local-min-lowerbd}, we denote
	$$
	\hat{\rS}_+=\{ i\in\hat{\rS}:G_i^\top \hat{x}>0 \}\subset\hat{\rS}
	$$
	and
	$$
	V(\hat{\rS},j,\hat{\rS}_+)=\{ v\in K(\hat{\rS}): G_i^\top v>0\ \forall i\in \hat{\rS}_+, G_i^\top v<0\ \forall i\in \hat{\rS}\backslash \hat{\rS}_+,\quad |G_j^\top v|=1 \}.
	$$
		Since $v=\frac{\hat{x}}{|G_j^\top \hat{x}|}\in V(\hat{\rS},j,\hat{\rS}_+)$, $
	V(\hat{\rS},j,\hat{\rS}_+)\neq\emptyset.
	$
	Let $\tilde{v}=\tilde{v}(\hat{\rS},j,\hat{\rS}_+)$ denote the solution to the following problem
	$$
	\min\|v\|_2\quad\text{subject to }v\in \overline{V(\hat{\rS},j,\hat{\rS}_+)},
	$$
	which satisfies
	$\|\tilde{v}\|_2=\|\tilde{v}(\hat{\rS},j,\hat{\rS}_+)\|_2\leq\mu$ with $\mu$ as defined in \eqref{eqn:mu}. Then,
	\begin{equation}\label{eqn:contr1}
	\phi^\prime(|G_j^\top \hat{x}|)=
	\phi^\prime(|G_j^\top \hat{x}|)\frac{G_j^\top \hat{x}}{|G_j^\top \hat{x}|}G_j
	^\top \tilde{v}\leq\sum_{i\in\hat{\rS}}\phi^\prime(|G_i^\top \hat{x}|)\frac{G_i^\top \hat{x}}{|G_i^\top \hat{x}|}G_i
	^\top \tilde{v}\leq\alpha_1(\tilde{x})\|\tilde{v}\|_2\leq\alpha_1(\tilde{x})\mu.
	\end{equation}
		Since $\phi^\prime(0+)=+\infty$, the constant $\theta=\inf\{ t>0:\phi^\prime(t)\leq\alpha_1(\tilde{x})\mu \}$ is well defined.
		If $\hat{x}$ is sufficiently close to $\tilde{x}$ such that $\|\hat{x}-\tilde{x}\|< \min\{\frac{\theta}{\max\limits_{i\in J}\|G_i\|},\|\tilde{x}\|\}$, we have
		$|G_j^\top \hat{x}|\leq |G_j^\top \hat{x}-G_j^\top  \tilde{x}| + |G_j^\top  \tilde{x}|\leq \|G_j\|\|\hat{x}-\tilde{x} \| <\theta$ and $\phi^\prime(|G_j^\top\hat{x}|)>\alpha_1(\tilde{x})\mu$, which  contradicts with \eqref{eqn:contr1}.
		Thus, for any $j\in \rJ$ with $G_j^\top \tilde{x}=0$,
		$
		G_i^\top \hat{x}=0.
		$
\end{proof}

%% file: section3.tex
The theoretical results in the previous section not only exhibit some interesting model properties, but also help us to design iterative algorithms and perform convergence analysis for the non-Lipschitz composite minimization model \eqref{eqn:lqlp}.
\subsection{Algorithms}

We now derive an iterative algorithm to solve the model \eqref{eqn:lqlp}.
In order to find a local minimizer or stationary point near to a given point, the support inclusion analysis in Theorem~\ref{thm:motivation} naturally motivates a so-called support shrinking strategy at each iteration, like \cite{zeng2019iterative,zeng2019nonlip,zheng2020a} for various image restoration models
with isotropic regularization.
Given $x^k$, we denote $\rS^k=\rS(x^k)=\{ i\in \rJ:G_i^\top x^k\neq0 \}$,
and define
$C_k=\{x:G_i^\top x=0,\ \forall i\in \rJ\backslash \rS^k \}$. Then for $k=0,1,2,\ldots$, the following
$$
\begin{cases}
\min\limits_{x\in\RR^N}
\left\{\mathcal{F}_k(x):=
\sum\limits_{i\in \rS^k}\phi(|G_i^\top x|)+\frac{\beta}{q}\|Ax-b\|_q^q \right\},\\
\text{s.t. }x\in C_k,
\end{cases}
$$
is considered to compute $x^{k+1}$.

However, considering the finite word length of real computers and to avoid extremely large linearization weights described later, we do not track the indices just in $\rS^k$ or use the constraints $\rJ\backslash \rS^k=\{i\in \rJ: G_i^\top x^{k}=0 \}$.
Instead, we relax these at each iteration by a thresholding operation, and propose the strategy of  iterative thresholding and support shrinking (ITSS). Indeed, this relaxation is actually usually adopted in computer codes for support shrinking based algorithms, but so far, has been formulated explicitly only in \cite{feng2020the} with incorporation into the iteratively reweighted least square (IRLS) algorithm for group sparse signal recovery (where $G=$Identity).
In particular, let us denote the $\tau$-support
$\rT^k=\rT(x^k)=\{ i\in \rJ:|G_i^\top x^k|>\tau \}$ and the set $C_k^\tau=\{x:G_i^\top x=0,\ \forall i\in \rJ\backslash \rT^k \}$
accordingly.
At the $k$'th iteration, the differences $G_i^\top x^{k+1}$ with $i\in \rJ\backslash \rT^k$, instead of
$i\in \rJ\backslash \rS^k$,  are constrained to be zero.
That is, we construct the following
$$
({\mathfrak{F}}_k^\tau)\quad
\begin{cases}
\min\limits_{x\in\RR^N}
\left\{{\mathcal{F}}_k^\tau(x):=
\sum\limits_{i\in \rT^k}\phi(|G_i^\top x|) + \frac{\beta}{q}\|Ax-b\|_q^q \right\},\\
\text{s.t. }x\in C_k^\tau.
\end{cases}
$$
This problem is still nonconvex and difficult.
We however can linearize the $\phi$ at $|G_i^\top x^k|$ with $i\in\rT^k$. With a proximal term, we then write the following strongly convex optimization problem \begin{equation}\label{eqn:hktau}
	{({\mathfrak{H}}_k^\tau)}
	\left\{
	\begin{aligned}
	&\min\limits_{x\in\RR^N} \left\{{\mathcal{H}}_k^\tau(x):=\sum_{i\in \rT^k}\left\{\phi(|G_i^\top x^{k}|)+\phi'(|G_i^\top x^{k}|)\left( |G_i^\top x|-|G_i^\top x^k| \right) \right\} +\frac{\beta}{q}\|Ax-b\|_q^q+\frac{\rho}{2}\|x-x^k\|^2 \right\},\\
	&\text{s.t. }x\in C_k^\tau,
	\end{aligned}
	\right.
	\end{equation}
	for computing $x^{k+1}$.
Obviously ${\mathcal{H}}_k^\tau(x)$ in \eqref{eqn:hktau} has a
	unique optimal solution.
	Thus, we have derived the \textit{Iterative thresholding and support shrinking algorithm with proximal linearization (ITSS-PL)} as stated in Algorithm~\ref{alg:itssapl}, which is with iteratively reweighted $\ell_1$ flavor.

\begin{algorithm}[htbp]
	\caption{Iterative thresholding and support shrinking algorithm with proximal linearization (ITSS-PL)}
	\begin{algorithmic}
		\REQUIRE $A$, $b$, $\beta>0$, $\rho>0$,   $\tau\geq0$,  \text{MAXit}, $x^0\in\RR^N$.
		\WHILE{$k\leq \text{MAXit} $}
		\STATE{Compute $x^{k+1}$ by solving $({\mathfrak{H}}_k^\tau)$}.
		\ENDWHILE
	\end{algorithmic}
	\label{alg:itssapl}
\end{algorithm}

Precisely solving the subproblem $({\mathfrak{H}}_k^\tau)$ still needs infinite iteration steps, which is not practical in real computation. Therefore, instead of finding the exact minimizer, we solve the subproblem $({\mathfrak{H}}_k^\tau)$ in Algorithm~\ref{alg:itssapl} to some given accuracy, like \cite{Attouch2013Convergence,zhang2017nonconvexTV,zheng2020a}.
For any set $C$, recall the indicator function $\kappa_C$ as
$$
\kappa_C(x)=
\begin{cases}
0&\quad x\in C,\\
+\infty&\quad\text{otherwise.}
\end{cases}
$$
Then we come to the \textit{inexact iterative thresholding and support shrinking algorithm with proximal linearization (Inexact ITSS-PL)}, as given in Algorithm~\ref{alg:initssapl}.
We mention that \textit{Inexact ITSS-PL} degenerates to \textit{ITSS-PL}, if $\epsilon=0$.

\begin{algorithm}[htbp]
	\caption{Inexact iterative thresholding and support shrinking algorithm with proximal linearization (Inexact ITSS-PL)}
	\begin{algorithmic}
		\REQUIRE $A\in\RR^{M\times N}$, $b\in\RR^N$, $\beta>0$, $\rho>0$,   $\tau\geq0$, $0<\epsilon<1$, \text{MAXit}, $x^0\in\RR^N$.
		\WHILE{$k\leq \text{MAXit} $}
		\STATE{Compute $x^{k+1}$ by
			\begin{numcases}{}
			x^{k+1}\approx
			\arg\min\limits_{x\in\RR^N} {\mathcal{H}}_k^\tau(x)+\kappa_{C_k^\tau}(x)
			\text{ and }h^{k+1}\in\partial\left({\mathcal{H}}_k^\tau(x^{k+1})
			+\kappa_{C_k^\tau}(x^{k+1})\right)	\notag\\
			\text{s.t. }\|h^{k+1}\|_2\leq \frac{\rho}{2}\epsilon \|x^{k+1}-x^k\|_2
			\label{eqn:inequk},
			\end{numcases}
			where $C_k^\tau=\left\{x:G_i^\top x=0,\ \forall i\in \rJ\backslash \rT^k \right\}$.}
		\ENDWHILE
	\end{algorithmic}
\label{alg:initssapl}
\end{algorithm}

\subsection{Convergence analysis}
In what follows, we prove that the sequence generated by Algorithm~\ref{alg:initssapl} with the inexact inner loop does converge to a stationary point of the original minimization model (${\mathfrak{F}}$) in \eqref{eqn:lqlp}.
The proof is based on the
Kurdyka-{\L}ojasiewicz (KL) property (\cite{Lojasiewicz1963Une,Kurdyka1998gradients}), which has been extensively applied to the analysis of various optimization methods \cite{absil2005convergence,bolte2007the,Attouch2009convergence,Attouch2010Proximal,Attouch2013Convergence,zhang2017nonconvexTV,zheng2020a}.

Note that
$$
\partial\kappa_{C_k^\tau}(x)=
\begin{cases}\sum\limits_{i\in\rJ\backslash\rT^k}(\ker G_i^\top)^\perp,& x\in C_k^\tau,\\
\emptyset,&\text{otherwise.}
\end{cases}
$$
By \eqref{eqn:inequk}, we then have
$$
G_i^\top x^{k+1}=0,\quad i\in \rJ\backslash \rT^k,
$$
which implies $\rJ\backslash \rS^k\subset \rJ\backslash \rT^k\subset \rJ\backslash \rS^{k+1} \subset \rJ\backslash \rT^{k+1}$ and $\rT^{k+1}\subset\rS^{k+1}\subset \rT^k \subset \rS^k\subset\cdots\subset\rJ$.
Due to the finiteness of the set $\rJ$, it is also straightforward to see the finite
convergence of both support and $\tau$-support sequences $\rS^k $ and $\rT^k $, as shown in Lemma \ref{thm:nested}, like the finite convergence of the support sequence mentioned in the literature.
\begin{lemma}[Finite convergence of support and $\tau$-support]\label{thm:nested} For any $k\geq0$, $\rT^{k+1}\subset\rS^{k+1}\subset \rT^k \subset \rS^k$ and the sequences $\{\rS^k \}$ and $\{\rT^k\}$ converge in a finite number of iterations, i.e., there exists an integer $K>0$ such that $\rS^k=\rT^k=\rS^K$ for any $k\geq K$.
\end{lemma}

By Lemma~\ref{thm:nested}, the sets $\rS^k$ and $\rT^k$ will be unchanged when $k\geq K$. In the following, to prove the convergence, we only need to focus on the case $k\geq K$. We denote
\begin{equation}\label{eqn:fixsupp}
\bar{\rS}=\rS^k=\rT^k=\{ i\in \rJ:G_i^\top x^k\neq0 \}=\{i\in\rJ: |G_i^\top x^k|>\tau \}\quad\text{for }k\geq K.
\end{equation}
and
$$
\bar{\mathcal{H}}_k^\tau(x):=
\sum\limits_{i\in \bar{\rS}}\left\{\phi(|G_i^\top x^{k}|)+\phi'(|G_i^\top x^{k}|)\left( |G_i^\top x|-|G_i^\top x^k| \right) \right\}
+ \frac{\beta}{q} \|Ax-b\|_q^q+\frac{\rho}{2}\|x-x^k\|^2.
$$
Then for $k\geq K$, the subproblem at the $k$'th step in Algorithm~\ref{alg:initssapl} becomes
\begin{equation}\label{eqn:fixmin}
\begin{cases}
x^{k+1}\approx
\arg\min\limits_{x\in\RR^N} \bar{\mathcal{H}}_k^\tau(x)
+\kappa_{\bar{C}}(x)
\text{ and }h^{k+1}\in\partial\left(\bar{\mathcal{H}}_k^\tau(x^{k+1})
+\kappa_{\bar{C}}(x^{k+1})
\right) \\
\text{s.t. }\|h^{k+1}\|_2\leq \frac{\rho}{2}\epsilon \|x^{k+1}-x^k\|_2
\end{cases},
\end{equation}
where $\bar{C}=\{x:G_i^\top x=0,\ \forall i\in \rJ\backslash \bar{\rS} \}$.

The next lemma is the decreasing property of the objective function.

\begin{lemma}[Sufficient decrease]\label{thm:decrease}
	The sequence $\{\mathcal{F}(x^k)\}_{k\geq K}$ is nonincreasing, and more precisely, for any $k\geq K$,
	$$
	(1-\epsilon)\frac{\rho}{2}\|x^{k+1}-x^k\|_2^2\leq \mathcal{F}(x^k)-\mathcal{F}(x^{k+1}).
	$$
\end{lemma}
\begin{proof}
	For any $k\geq K$, by the convexity of function $\bar{\mathcal{H}}_k^\tau$ in \eqref{eqn:fixmin}, we have
	\begin{align*}
	\bar{\mathcal{H}}_k^\tau(x^{k+1})=&\sum_{i\in \bar{\rS}}
\left\{	\phi(|G_i^\top x^{k}|)+\phi'(|G_i^\top x^{k}|)\left( |G_i^\top x^{k+1}|-|G_i^\top x^k| \right)\right\}
+\frac{\beta}{q}\|Ax^{k+1}-b\|_q^q+\frac{\rho}{2}\|x^{k+1}-x^k\|_2^2\\
	\leq&\bar{\mathcal{H}}_k^\tau(x^k) - \langle h^{k+1},x^k-x^{k+1}\rangle\\
	\leq & \sum_{i\in \bar{\rS}}\phi(|G_i^\top x^{k}|)+\frac{\beta}{q}\|Ax^{k}-b\|_q^q +\frac{\rho}{2}\epsilon\|x^{k+1}-x^{k}\|_2^2\\
	=& \mathcal{F}(x^k) +\frac{\rho}{2}\epsilon\|x^{k+1}-x^{k}\|_2^2.
	\end{align*}
		By Assumption \ref{thm:xp},
		\begin{equation}\label{eqn:concave}
		\phi(t)\leq\phi(\bar{t})+\phi^\prime(\bar{t})(t-\bar{t}),\quad\forall t\geq0\text{ and } \bar{t}>0.
		\end{equation}
	Then by \eqref{eqn:concave},
	\begin{align*}
	\bar{\mathcal{H}}_k^\tau(x^{k+1})
	\geq & \sum_{i\in \bar{\rS}}\phi(|G_i^\top x^{k+1}|) +\frac{\beta}{q}\|Ax^{k+1}-b\|_q^q+\frac{\rho}{2}\|x^{k+1}-x^k\|_2^2\\
	= & \mathcal{F}(x^{k+1}) + \frac{\rho}{2}\|x^{k+1}-x^k\|_2^2.
	\end{align*}
	Therefore, we obtain
	$$
	\mathcal{F}(x^k)- \mathcal{F}(x^{k+1}) \geq (1-\epsilon)\frac{\rho}{2}\|x^{k+1}-x^k\|_2^2\quad\text{for any }k\geq K.
	$$
	\end{proof}

From Lemma~\ref{thm:decrease}, we see the boundness of the sequence $\{ x^k\}$.

\begin{lemma}[Square summable and asymptotic convergence]\label{thm:sum}
	The sequence $\{x^k\}$ is bounded and satisfies
	$$
	\sum_{k=0}^\infty \|x^{k+1}-x^k\|^2<\infty.
	$$
	Hence $\lim\limits_{k\to\infty}\|x^{k+1}-x^k\|=0$.
\end{lemma}
\begin{proof}
	By Lemma~\ref{thm:decrease}, $\{ \mathcal{F}(x^k) \}$ is bounded and convergent. Since $\mathcal{F}(x)$ is coercive (Theorem \ref{thm:converge}), $\{x^k\}$ is bounded.	
	Again by Lemma~\ref{thm:decrease}, we get
	$$
	(1-\epsilon)\frac{\rho}{2}\sum_{k=0}^{K_1}\|x^{k+1}-x^k\|^2\leq\mathcal{F}(x^0)-\mathcal{F}(x^{K_1+1})\leq\mathcal{F}(x^0).
	$$
	Letting $K_1\to\infty$ completes the proof.
\end{proof}

A very useful lemma is the following uniform lower bound result of nonzero differences of the iterate sequence. 

\begin{lemma}[Lower bound of the sequence]\label{thm:lb-sequence}
	There exists a constant $\bar{\theta}>0$ such that
	$$
	|G_i^\top x^k|\geq\bar{\theta}\quad\forall k\geq K,\forall i\in\bar{\rS}.
	$$
\end{lemma}
\begin{proof}
	Without loss of generality, suppose that $\bar{\rS}\neq\emptyset$.
	By \eqref{eqn:fixsupp}, a natural and obvious consequence is
	$$
	|G_i^\top x^k|>\tau\geq0\quad\forall k\geq K,\forall i\in\bar{\rS}.
	$$
    Next, we derive a positive lower bound of nonzero differences of the iterate sequence, even if $\tau=0$.
	
	Since $\kappa_{\bar{C}}(x)$ 
	is subdifferentially regular at $x\in\bar{C}$, we derive
	\begin{align*}
	h^{k+1}\in &\partial \left(\bar{\mathcal{H}}_k(x^{k+1}) +
	\kappa_{\bar{C}}(x^{k+1})
	\right)\\
	= & \partial\left(\frac{\beta}{q}\|Ax^{k+1}-b\|_q^q \right)
	+ \sum_{i\in\bar{\rS}}\phi^\prime(|G_i^\top x^{k}|)\frac{G_i^\top x^{k+1}}{|G_i^\top x^{k+1}|}G_i+\rho(x^{k+1}-x^{k})
	+ \partial\kappa_{\bar{C}}(x^{k+1}),
	\end{align*}
	where
	$\partial\left(\frac{\beta}{q}\|Ax^{k+1}-b\|_q^q \right)
	=\frac{\beta}{q}A^\top\partial \|\cdot\|_q^q (A\hat{x}-b)
	$
	  and $\partial\kappa_{\bar{C}}(x^{k+1})
	=\sum\limits_{i\in\rJ\backslash\bar{\rS}}(\ker G_i^\top)^\perp$.
Then there exist
$z^{k+1}\in \partial\left(\frac{\beta}{q}\|Ax^{k+1}-b\|_q^q \right)$
	and $\eta_i^{k+1}\in(-\infty,+\infty)$ ($i\in\rJ\backslash\bar{\rS}$) such that
	$$
	h^{k+1}=
	z^{k+1}
	+ \sum_{i\in\bar{\rS}}\phi^\prime(|G_i^\top x^{k}|)\frac{G_i^\top x^{k+1}}{|G_i^\top x^{k+1}|}G_i+\rho(x^{k+1}-x^{k}) + \sum_{i\in\rJ\backslash\bar{\rS}}\eta_i^{k+1}G_i.
	$$

	Denote $K(\bar{\rS})=\left\{ x\in\RR^N:\  G_i^\top x=0\ \forall i\in \rJ\backslash \bar{\rS} \right\}$ and $\bar{\rS}^{k+1}_+=\{i\in\bar{\rS}:G_i^\top x^{k+1}>0  \}$. Consider a fixed index $i^\prime\in \bar{\rS}$ and define the set
	$$
	V(\bar{\rS},i^\prime,\bar{\rS}^{k+1}_+)=\{ v\in K(\bar{\rS}): G_i^\top v>0\ \forall i\in \bar{\rS}^{k+1}_+,\quad G_i^\top v<0\ \forall i\in \bar{\rS}\backslash \bar{\rS}^{k+1}_+, \quad |G_{i^\prime}^\top v|=1 \}.
	$$
	Note that $V(\bar{\rS},i^\prime,\bar{\rS}^{k+1}_+)\neq\emptyset$, since $v=\frac{x^{k+1}}{|G_{i^\prime}^\top x^{k+1}|}\in V(\bar{\rS},i^\prime,\bar{\rS}^{k+1}_+)$. Take $v=\tilde{v}^{k+1}$ to be the solution to the following problem
	$$
	\min\|v\|_2\quad\text{s.t. }v\in \overline{V(\bar{\rS},i^\prime,\bar{\rS}^{k+1}_+)}
	$$
	implying that
	$$
	\tilde{v}^{k+1}\in K(\bar{\rS}); \text{ for any } i\in\bar{\rS},\text{ either } G_i^\top x^{k+1}\cdot G_i^\top \tilde{v}^{k+1}>0\text{ or }G_i^\top \tilde{v}^{k+1}=0;\ |G_{i^\prime}^\top \tilde{v}^{k+1}|=1,
	$$
	and $\|\tilde{v}^{k+1}\|_2\leq\mu$ with $\mu$ as defined in \eqref{eqn:mu}. It then follows that
	\begin{align*}
	\langle h^{k+1},\tilde{v}^{k+1}\rangle
	= \langle z^{k+1}, \tilde{v}^{k+1}\rangle
	+ \sum_{i\in\bar{\rS}}\phi^\prime(|G_i^\top x^{k_j}|)\frac{G_i^\top x^{k+1}}{|G_i^\top x^{k+1}|}G_i^\top \tilde{v}^{k+1} \\ +\rho\langle x^{k+1}-x^{k},\tilde{v}^{k+1}\rangle + \sum_{i\in\rJ\backslash\bar{\rS}}\eta_i^{k+1}G_i^\top \tilde{v}^{k+1},
	\end{align*}
	indicating
	\begin{align*}
	\langle h^{k+1},\tilde{v}^{k+1}\rangle -
	\langle z^{k+1},\tilde{v}^{k+1}\rangle
	- \rho\langle x^{k+1}-x^{k},\tilde{v}^{k+1}\rangle
	&=\sum_{i\in\bar{\rS}}\phi^\prime(|G_i^\top x^{k}|)\frac{G_i^\top x^{k+1}}{|G_i^\top x^{k+1}|}G_i^\top \tilde{v}^{k+1}\\
	&\geq\phi^\prime(|G_{i^\prime}^\top x^{k+1}|).
	\end{align*}
	On the other hand, by the boundness of $\{x^k\}$, there exists a constant $\Gamma>0$ such that $\|x^k\|_2\leq\Gamma$ for any $k\in\mathbb{N}$. Thus, the above inequality becomes
	\begin{align*}
	\phi^\prime(|G_{i^\prime}^\top x^{k+1}|)\leq
	\langle h^{k+1},\tilde{v}^{k+1}\rangle -
	\langle z^{k+1},\tilde{v}^{k+1}\rangle
	- \rho\langle x^{k+1}-x^{k},\tilde{v}^{k+1}\rangle\leq \bar{\alpha}
	\end{align*}
	where $\bar{\alpha}=
	(\epsilon+2)\mu\rho\Gamma
	+\beta\mu\sum\limits_{l=1}^M \left(\Gamma\| (A^\top)_l\|_2+|b_l|\right)^{q-1} \|(A^\top)_l\|_2$.	
	Since $\phi^\prime(0+)=+\infty$, the constant $\theta'=\inf\{ t>0:\phi^\prime(t)\leq\bar{\alpha} \}>0$ is well defined and consequently
	$$
	|G_{i'}^\top  x^{k+1}|\geq\theta^\prime,
	$$
	holds for any $i'\in \bar{\rS}$ and any $k\geq K$.

	Therefore, the uniform lower bound of nonzero differences of the iterate sequence can be taken as
	$$
	\bar{\theta}
	=\max\left\{ \tau,\theta' \right\}>0.
	$$
\end{proof}
	
	In Lemma~\ref{thm:lb-sequence}, we have proved the existence of a uniform lower bound for nonzero differences generated by the iterative sequence. Meanwhile, as stated in Assumption~\ref{thm:xp} (3), $\phi^\prime$ is $L_
	\alpha$-Lipschitz continuous on $[\alpha,+\infty)$ for any $\alpha>0$. We can then overcome the difficulties in the convergence analysis brought by  the non-Lipschitz property of $\phi^\prime$ at the zero point, and
	a subgradient lower bound for the iteration gap is given in the following lemma.

\begin{lemma}[A subgradient lower bound for the iteration gap]\label{thm:lb-subgradient}
	For each $k\geq K$, there exists $s^{k+1}\in\partial \mathcal{F}(x^{k+1})$ such that
	$$
	\|s^{k+1}\|_2\leq c\|x^{k+1}-x^{k}\|_2,
	$$
	for some constant $c$.
\end{lemma}
\begin{proof}
		It is known that  $\phi(|G_i^\top x|)$ is subdifferentially regular at all the points including $x^{k+1}$ for any $i\in \rJ$.
		Meanwhile,
		$\|Ax-b\|_q^q$ is also subdifferentially regular.
		We therefore have
	\begin{align*}
	\partial \mathcal{F}(x^{k+1})
&= \partial\left(\frac{\beta}{q} \|Ax^{k+1}-b\|_q^q \right) + \sum_{i\in\bar{\rS}} \phi^\prime(|G_i^\top x^{k+1}|)\frac{G_i^\top x^{k+1}}{|G_i^\top x^{k+1}|}G_i + \partial\left( \sum_{i\in\rJ\backslash\bar{\rS}}\phi(|G_i^\top x^{k+1}|) \right) \\
	&=  \partial\left(\frac{\beta}{q} \|Ax^{k+1}-b\|_q^q \right) + \sum_{i\in\bar{\rS}} \phi^\prime(|G_i^\top x^{k+1}|)\frac{G_i^\top x^{k+1}}{|G_i^\top x^{k+1}|}G_i+\sum_{i\in\rJ\backslash\bar{\rS}}(\ker G_i^\top)^\perp.
	\end{align*}

	By Algorithm~\ref{alg:initssapl}, $\|h^{k+1}\|_2\leq\frac \rho 2 \epsilon\|x^{k+1}-x^k\|_2$ and
	\begin{align*}
	h^{k+1}\in
	 & \partial\left(\frac{\beta}{q}\|Ax^{k+1}-b\|_q^q\right) + \sum_{i\in\bar{\rS}}\phi^\prime(|G_i^\top x^{k}|)\frac{G_i^\top x^{k+1}}{|G_i^\top x^{k+1}|}G_i+\rho(x^{k+1}-x^{k})
	+ \partial \kappa_{\bar{C}}(x^{k+1}),
	\end{align*}
	where  $\partial\kappa_{\bar{C}}(x^{k+1})=\sum\limits_{i\in\rJ\backslash\bar{\rS}}(\ker G_i^\top)^\perp
	$.
	
It is then easy to see
	$$
	s^{k+1}=h^{k+1}
	+\sum_{i\in\bar{\rS}} \left( \phi^\prime(|G_i^\top x^{k+1}|) - \phi^\prime(|G_i^\top x^{k}|) \right) \frac{G_i^\top x^{(k+1)}}{|G_i^\top x^{(k+1)}|}G_i-\rho(x^{k+1}-x^k),
	$$
	is in $\partial \mathcal{F}(x^{k+1})$.
By Assumption~\ref{thm:xp} (c) 
	and Lemma~\ref{thm:lb-sequence}, 
	we further derive
	\begin{align*}
	\|s^{k+1}\|_2 & \leq
	\|h^{k+1}\|_2
	+\sum_{i\in\bar{\rS}} | \phi^\prime(|G_i^\top x^{k+1}|) - \phi^\prime(|G_i^\top x^{k}|) |\|G_i\|_2 + \rho\|x^{k+1}-x^{k}\|_2 \\
	&  \leq \sum_{i\in\bar{\rS}} L_{\bar{\theta}} \left| |G_i^\top x^{k+1}| - |G_i^\top x^{k}| \right|\cdot \|G_i\|_2 + \rho(1+\frac\epsilon 2)\|x^{k+1}-x^{k}\|_2 \\
	& \leq  \left(  L_{\bar{\theta}} \sum_{i\in\rJ} \|G_i\|^2 + \rho(1+\frac\epsilon 2) \right) \|x^{k+1}-x^{k}\|_2.
	\end{align*}
	Taking $c= L_{\bar{\theta}} \sum\limits_{i\in\rJ} \|G_i\|^2 + \rho(1+\frac\epsilon 2) $ completes the proof.
\end{proof}

We can now establish the global convergence of the sequence generated by Inexact ITSS-PL (Algorithm~\ref{alg:initssapl}).

	\begin{theorem}[Global convergence]
		\label{thm:converge}
		Suppose that $\mathcal{F}(x)$ is a KL function.
		The sequence $\{ x^k \}$ generated by Algorithm~\ref{alg:initssapl} converges to a stationary point of $(\mathfrak{F})$.
	\end{theorem}
	\begin{proof}
		By Lemma~\ref{thm:sum}, $\{x^k\}$ is bounded. Then there exists a subsequence $\{x^{k_j}\}$ and a point $x^\ast$ such that
		$$
		x^{k_j}\to x^\ast\text{ and }\mathcal{F}(x^{k_j})\to\mathcal{F}(x^\ast)\quad\text{as }j\to\infty.
		$$
		Since $\mathcal{F}(x)$ is a KL function, by Lemma~\ref{thm:decrease}, Lemma~\ref{thm:lb-subgradient} and \cite[Theorem 2.9]{Attouch2013Convergence}, the sequence $\{x^k\}$ converges globally to $x^\ast$, which is a stationary point of problem $(\mathfrak{F})$.
	\end{proof}
	\begin{remark}
		A proper lower semicontinous function satisfying the KL property at all points in its domain is called a KL function.
		The objective functions $\mathcal{F}(x)$ in our examples are KL functions.
		Indeed, it is known that any proper lower semicontinuous function that is definable on an o-minimal structure is a KL function. See~\cite{bolte2007clarke} and~\cite[Theorem 4.1]{Attouch2010Proximal}. A class of o-minimal structure is the log-exp structure (\cite[Example 2.5]{Lou1996Geometric}). By this structure,
		the examples of potential function $\phi(t)=t^p$ and $\phi(t)=\log(1+t^p)$ are definable, and $|G_i^\top x|$, as a semi-algebraic function, is also definable.
		Then the composition $\phi(|G_i^\top x|)$ and the finite sum $\sum\limits_{i\in\rJ}\phi(|G_i^\top x|)$ are definable functions.
		Similarly, the fidelity $\frac{\beta}{q}\|Ax-b\|_q^q$ is a definable function.
		As the sum of regularization and fidelity terms, the objective function $\mathcal{F}(x)$ is definable, therefore it is a KL function.
	For the related preliminaries, one may refer to the appendix.
	\end{remark}

%% file: section4.tex
The subproblem $({\mathfrak{H}}_k^\tau)$ at each iteration of Inexact ITSS-PL can be solved by many effective convex optimization algorithms.
We consider to use the Alternating Direction Method of Multipliers (ADMM \cite{glowinski1989augmented,Wu2010Augmented,boyd2011distributed,he2012on,glowinski2016splitting}) and
the followings are the implementation details.

In the description of ADMM, the variable $x$ in the subproblem $({\mathfrak{H}}_k^\tau)$ is replaced by $u$ to avoid confusion with the outer iterations.
We first consider the general form of $({\mathfrak{H}}_k^\tau)$ with $q\in[1,+\infty)\backslash\{2\}$ in the fidelity.
To solve a such subproblem,
we introduce the variables $v\in\RR^M$ and $\{w_{i}\}_{i\in \rT^{k}}\in\RR^{|\rT^k|}$, and reformulate the subproblem $({\mathfrak{H}}_k^\tau)$ as the following equivalent form,
\begin{equation}
\begin{aligned}
& \min_{u,v,\{w_{i}\}_{i\in \rT^{k}}}   \sum_{i\in \rT^{k}} \phi^{'}(|G_{i}^{\top}x^{k}|)|w_{i}|
+ \beta \|v\|_q^q
+ \frac{1}{2}\rho \left\| u-x^{k} \right\|_2^2   \\
& \qquad\text{s.t.}
\begin{cases}
G_{i}^{\top}u=0,  &\forall i\in \rJ\backslash\rT^{k}  \\
G_{i}^{\top}u=w_{i},  &\forall i\in \rT^{k}  \\
Au-b=v.
\end{cases}
\end{aligned}
\end{equation}
The augmented Lagrangian function of the above constrained problem reads
\begin{equation}
\begin{aligned}
& \mathcal{L}(u, v, \{w_{i}\}_{i\in \rT^{k}}, \lambda_v, \lambda_w; r_v, r_w)  \\
= & \left. \beta \|v\|_q^q + \sum_{i\in \rT^{k}} \phi^{'}(|G_{i}^{\top}x^{k}|)|w_{i}| + \frac{1}{2}\rho \left\| u-x^{k} \right\|_2^2 \right.
\left. + \langle \lambda_v, Au-b-v \rangle + \frac{1}{2}r_v\left\| Au-b-v \right\|_2^2 \right.  \\
& \left. + \sum_{i\in \rT^{k}} \langle (\lambda_w)_{i}, G_{i}^{\top}u-w_{i} \rangle + \frac{1}{2}r_w\sum_{i\in \rT^{k}} |G_{i}^{\top}u-w_{i}|^2 \right.
\left. + \sum_{i\in \rJ\backslash \rT^{k}} \langle (\lambda_w)_{i}, G_{i}^{\top}u \rangle + \frac{1}{2}r_w\sum_{i\in \rJ\backslash \rT^{k}} |G_{i}^{\top}u|^2 \right.  
\end{aligned}
\end{equation}
where $\lambda_v\in\RR^M$ and $\lambda_w\in\RR^{\sharp\rJ}$ are the Lagrangian multipliers and $r_v,r_w>0$.
Then the ADMM to solve $({\mathfrak{H}}_k^\tau)$ is given in Algorithm $\ref{algo:ADMM-L1ITSSAPL}$.
\begin{algorithm}[htbp]
	\caption{ADMM to solve the subproblem $({\mathfrak{H}}_k^\tau)$ with $q\in[1,+\infty)\backslash\{2\}$} \label{algo:ADMM-L1ITSSAPL}
	\begin{algorithmic}
		\REQUIRE 
		$r_v>0$, $r_w>0$, MAXit1, $\epsilon_1>0$,
		$u^{0}=x^{k}$, $\lambda_v^{0}=0\in\RR^{M}$,
		$\lambda_w^0=0\in\RR^{\sharp\rJ}$
		\WHILE{$l\leq$MAXit1 and $\frac{\left\|u^{l+1}-u^{l}\right\|_2}{\left\|u^{l+1}\right\|_2} > \epsilon_{1}$}
		\STATE 1. Compute $(v^{l+1},\{w_{i}^{l+1}\}_{i\in \rT^{k}})$ by
		$\min_{v,\{w_{i}\}_{i\in \rT^{k}}} \{\mathcal{L}(u^{l},v,\{w_{i}\}_{i\in \rT^{k}},\lambda_v^{l},\lambda_w^{l};r_v,r_w)\}$;
		\STATE 2. Compute $u^{l+1}$ by
		$\min_{u} \{\mathcal{L}(u,v^{l+1},\{w_{i}^{l+1}\}_{i\in \rT^{k}},\lambda_v^{l},\lambda_w^{l};r_v,r_w)\}$;
		\STATE 3. Update $\lambda_v^{l+1}$ and $\lambda_w^{l+1}$ by
		$\lambda_v^{l+1} = \lambda_v^{l}+r_v(Au^{l+1}-b-v^{l+1})$ and
		$$
		(\lambda_w^{l+1})_{i} =
		\begin{cases}
		(\lambda_w^l)_{i}+r_w(G_{i}^{\top}u^{l+1}-w_{i}^{l+1}), & \text{if} \quad i\in \rT^{k},  \\
		(\lambda_w^l)_{i}+r_w G_{i}^{\top}u^{l+1}, & \text{if} \quad i\in \rJ\backslash \rT^{k}.
		\end{cases}
		$$
		\ENDWHILE
	\end{algorithmic}
\end{algorithm}

	For the ($v$, $\{w_{i}\}_{i\in \text{T}^{k}}$)-subproblem in Algorithm \ref{algo:ADMM-L1ITSSAPL}, $\{w_{i}^{l+1}\}_{i\in \rT^{k}}$ has the explicit form solution given by the soft-thresholding (\cite{donoho1994ideal,Donoho1995DenoisingBS}):
	\begin{equation*}
	w_{i}^{l+1}
	 = \max\left\{ |G_{i}^{\top}u^{l} + \frac{1}{r_{w}}(\lambda_{w}^{l})_{i}| - \frac{1}{r_{w}}\phi^{'}(|G_{i}^{\top}x^{k}|), 0\right\}\cdot \text{sign}\left(G_{i}^{\top}u^{l} + \frac{1}{r_{w}}(\lambda_{w}^{l})_{i}\right),\forall i\in \text{T}^{k},
	\end{equation*}
	and $v^{l+1}$ can be solved by numerical algorithms from the following convex optimization problem
	\begin{equation}
	\min_{v} \left\{ \beta\|v\|_q^q + \langle \lambda_{v}, Au^{l}-b-v \rangle + \frac{1}{2}r_{v}\left\| Au^{l}-b-v \right\|_2^2 \right\}.
	\end{equation}
	If $q=1$, $v^{l+1}$ also has the explicit form solution:
	\begin{equation*}
	v_{j}^{l+1}
	=  \max\left\{ |(Au^{l})_j-b_{j}+\frac{1}{r_{v}}(\lambda_{v}^{l})_{j}| - \frac{\beta}{r_{v}},0 \right\}\cdot \text{sign}\left((Au^{l})_j-b_{j}+\frac{1}{r_{v}}(\lambda_{v}^{l})_{j}\right), \forall 1\leq j\leq M.
	\end{equation*}
where $(Au^l)_j$ is the $j$'th element of vector $Au^l$.

	For solving the $u$-subproblem, we augment $\{w_{i}^{l+1}\}_{i\in \rT^{k}}$ by $w_{i}^{l+1} = 0$, $\forall i\in \text{J}\backslash\text{T}^{k}$, so that we can make full use of the structure of the operators $\{G_{i}\}_{i\in\text{J}}$, like \cite{zeng2019iterative,zheng2020a}. Then the $u$-subproblem can be equivalent to
	\begin{align}\label{u-subproblem-formula}
	\begin{split}	&\min_{u} \left\{ \frac{1}{2}u^{\top}\left(\rho I_{N} + r_{v} A^{\top}A + r_{w}\sum_{i\in \rJ} G_{i}G_{i}^{\top}\right)u\right.  \\
	& \left.\qquad\qquad  - \left(\rho x^{k} + A^{\top}(r_{v} b + r_{v}v^{l+1} - \lambda_{v}^{l}) + \sum_{i\in \text{J}} \left(r_{w}w_{i}^{l+1}-(\lambda_{w}^{l})_{i}\right)G_{i}\right)^{\top}u  \right\},
	\end{split}
	\end{align}
	whose solution can be obtained by solving the corresponding normal equation.
	In the case of convolutional operators $A,G$ with periodic boundary condition raised in image-deblur like applications, $\rho I_{N} + r_{v} A^{\top}A + r_{w}\sum_{i\in \rJ} G_{i}G_{i}^{\top}$
	is a block circulant matrix
	and it can be diagonalized by the
	two-dimensional discrete Fourier transforms. The $u$-subproblem (\ref{u-subproblem-formula}) can then be efficiently solved by utilizing the  fast Fourier transform, like \cite{Wang2008,Wu2010Augmented}.

We are then left with the special case of $({\mathfrak{H}}_k^\tau)$ with $q=2$ in the fidelity, i.e., the case of quadratic fidelity term. This is an easier case and the subproblem $({\mathfrak{H}}_k^\tau)$ can be reformulated into
\begin{equation}
\begin{aligned}
& \min_{u,\{w_{i}\}_{i\in \rT^{k}}}   \sum_{i\in \rT^{k}} \phi^{'}(|G_{i}^{\top}x^{k}|)|w_{i}|
+ \beta \|Au-b\|_2^2
+ \frac{1}{2}\rho \left\| u-x^{k} \right\|_2^2   \\
& \qquad\text{s.t.}
\begin{cases}
G_{i}^{\top}u=0,  &\forall i\in \rJ\backslash\rT^{k}  \\
G_{i}^{\top}u=w_{i},  &\forall i\in \rT^{k}.
\end{cases}
\end{aligned}
\end{equation}
whose corresponding augmented Lagrangian function
\begin{equation}
\begin{aligned}
\mathcal{L}(u, \{w_{i}\}_{i\in \rT^k}, \lambda_w; r_w)
= &  \beta\left\| Au-b \right\|_2^2 + \sum_{i\in \rT^k} \phi^{'}(|G_{i}^{\top}x^{k}|)|w_{i}| + \frac{\rho}{2} \left\| u-x^{k} \right\|_2^2
+ \sum_{i\in \rT^{k}} \langle (\lambda_w)_{i}, G_{i}^{\top}u-w_{i} \rangle \\
& + \frac{1}{2}r_w\sum_{i\in \rT^{k}} |G_{i}^{\top}u-w_{i}|^2
+ \sum_{i\in \rJ\backslash \rT^{k}} \langle (\lambda_w)_{i}, G_{i}^{\top}u \rangle + \frac{1}{2}r_w\sum_{i\in \rJ\backslash \rT^{k}} |G_{i}^{\top}u|^2,
\end{aligned}
\end{equation}
where $\{w_{i}\}_{i\in \rT^k}\in\RR^{|\rT^k|}$ is the auxiliary variable, $\lambda_w\in\RR^{\sharp\rJ}$ is the Lagrangian multiplier and $r_w>0$.
The ADMM to solve $({\mathfrak{H}}_k^\tau)$ with $q=2$
is shown in Algorithm $\ref{algo:ADMM-L2ITSSAPL}$.
\begin{algorithm}[htbp]
	\caption{ADMM to solve the subproblem $({\mathfrak{H}}_k^\tau)$ with $q=2$}
	\label{algo:ADMM-L2ITSSAPL}
	\begin{algorithmic}
		\REQUIRE  
		$r_w>0$, MAXit2, $\epsilon_2>0$,
		$u^{0}=x^{k}$, $\lambda_w^0=0\in\RR^{\sharp\rJ}$
		\WHILE{$l\leq$MAXit2 and $\frac{\left\|u^{l+1}-u^{l}\right\|_2}{\left\|u^{l+1}\right\|_2} > \epsilon_{2}$}
		\STATE 1. Compute $\{w_{i}^{l+1}\}_{i\in \rT^{k}}$ by
		$\min_{\{w_{i}\}_{i\in \rT^{k}}} \{\mathcal{L}(u^{l},\{w_{i}\}_{i\in \rT^{k}},\lambda_w^{l};r_w)\}$;
		\STATE 2. Compute $u^{l+1}$ by
		$\min_{u} \{\mathcal{L}(u,\{w_{i}^{l+1}\}_{i\in \rT^{k}},\lambda_w^{l};r_w)\}$;
		\STATE 3. Update $\lambda_w^{l+1}$ by
		$$
		(\lambda_w^{l+1})_{i} =
		\begin{cases}
		(\lambda_w^{l})_{i}+r_w(G_{i}^{\top}u^{l+1}-w_{i}^{l+1}), & \text{if} \quad i\in \rT^{k},  \\
		(\lambda_w^{l})_{i}+r_w G_{i}^{\top}u^{l+1}, & \text{if} \quad i\in \rJ\backslash \rT^{k}.
		\end{cases}
		$$
		\ENDWHILE
	\end{algorithmic}
\end{algorithm}

\begin{remark}
		Since the ADMM converges to the unique minimizer (with zero subgradient) of the strongly convex $({\mathfrak{H}}_k^\tau)$,
	the condition \eqref{eqn:inequk} on the subdifferential $h^{k+1}$ in Algorithm \ref{alg:initssapl} is satisfied, after a large enough number of iterations of ADMM and a projection to the feasible set. However, checking condition \eqref{eqn:inequk} at each iteration of the subsolver is time consuming.
	Therefore, in practical computation, we propose a simple stopping criterion, that is to check whether it satisfies  $\frac{\left\|u^{l+1}-u^{l}\right\|_2}{\left\|u^{l+1}\right\|_2} < \epsilon_{1}\ (\text{or } \epsilon_2)$ and whether the iteration number $l$ exceeds the predefined maximum iteration number MAXit1 (or MAXit2).
Such an approach can save running time and achieve fairly good restoration results.
\end{remark}

%% file: section5.tex
In this section, we test the performance of the proposed algorithm in image deconvolution, and specifically, we consider two types of noises: 1) the
salt and pepper impulse noise and 2) the i.i.d.
Gaussian noise.
Therefore,
we recall here the model \eqref{eqn:model1} and model \eqref{thm:lpl2}, both of which
can be solved by the Inexact ITSS-PL given in Algorithm \ref{alg:initssapl}.
For convenience, we rename the Inexact ITSS-PL for model \eqref{eqn:model1} as
\textit{Inexact iterative thresholding and support shrinking algorithm with proximal linearization for $\ell_1$ fidelity (Inexact ITSS-PL-$\ell_1$)}, and rename the ITSS-PL for model \eqref{thm:lpl2} as
\textit{Inexact iterative thresholding and support shrinking algorithm with proximal linearization for $\ell_2$ fidelity (Inexact ITSS-PL-$\ell_2$)}.

In the experiments, we take the potential function $\phi$ as $\phi(t) := t^{p}$, $\forall t\in\mathbb{R}$, with $p\in(0,1)$, and take the sparsifying system $\{G_i\}_{i\in\rJ}$ as the horizontal and vertical discrete derivative operators.
In this case, model \eqref{eqn:model1} and model \eqref{thm:lpl2} become the anisotropic
$\ell_{1}\text{TV}^{p}$ and $\ell_{2}\text{TV}^{p}$ models (namely the $\ell_{1}\text{aTV}^{p}$ and $\ell_{2}\text{aTV}^{p}$ models), and they are respectively solved by the Inexact ITSS-PL-$\ell_1$ and Inexact ITSS-PL-$\ell_2$. 
The subproblems  $({\mathfrak{H}}_k^\tau)$ in Inexact ITSS-PL-$\ell_1$ and Inexact ITSS-PL-$\ell_2$ are  solved by the ADMM given in Algorithm~\ref{algo:ADMM-L1ITSSAPL} and  Algorithm~\ref{algo:ADMM-L2ITSSAPL}, respectively.

\subsection{Test platform and parameter choices}\label{subsec:parameters}

The experiments are performed under Windows 8 and MATLAB R2018a running on a desktop equipped with an Intel Core i7-6700 CPU @ 3.40GHz and 8.00G RAM memory.

For the experiments of image deconvolution, the test images are first degraded by a blur kernel and
then added by some noise.
Three types of blurring kernels are used, which include (1) averaging blur 
(fspecial('average',5));
(2) Gaussian blur 
(fspecial('gaussian',[23,23],12));
(3) disk blur 
(fspecial('disk',6)).
Two types of noises are considered:
(1) salt-and-pepper impulse noise with noise level $30\%$; (2) i.i.d. Gaussian white noise with variance $10^{-6}$.
The quality of the restored images is measured by
the peak signal-to-noise ratio (PSNR), which is defined as follows:
\begin{equation}
\text{PSNR} = 10  \log_{10}\frac{N}{\|x_{\text{Alg}} -\underline{x}\|_2^2},  \notag
\end{equation}
where $N$ is the number of image pixels, $x_{\text{Alg}} $ is the restored image, and $\underline{x}$ is the true image.
The test images are shown in Figure \ref{figs-origin}.

\begin{figure}[htbp]
	\subfigure[]
	{
		\includegraphics[width=2.5cm]{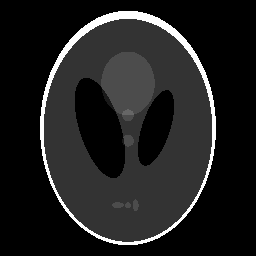}
	}
\hspace{-10pt}
	\subfigure[]
	{
		\includegraphics[width=2.5cm]{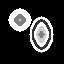}
	}
\hspace{-10pt}
	\subfigure[]
	{
		\includegraphics[width=2.5cm]{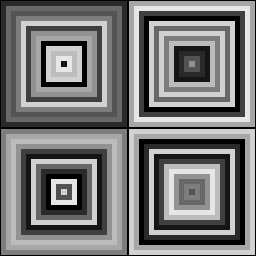}
	}
\hspace{-10pt}
	\subfigure[]
	{
		\includegraphics[width=2.5cm]{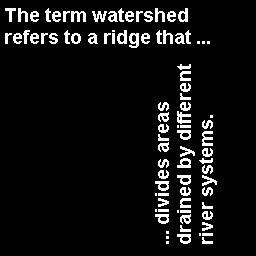} 
	}
\hspace{-10pt}
	\subfigure[]
	{
		\includegraphics[width=2.5cm]{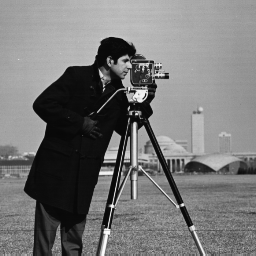}
	}
	\caption{
			Test images for restoration problem.
			(a): Phantom ($256\times 256$); (b): Twocircles ($64\times64$)
			(c): Squares ($256\times 256$); (d): Text ($256\times 256$);
			(e): Cameraman ($256\times 256$).}
	\label{figs-origin}
\end{figure}

For the model parameters,
	we use $p=0.5$ in the $\ell_{q}\text{aTV}^{p}\ (q=1,2)$ model,
	and the parameter $\beta$ is tuned up to achieve the best performance of each method, which will be specified in the subsequent experiments.
For the algorithm parameters, they will be set as follows and remain unchanged throughout the experiments, unless otherwise specified.
For both Inexact ITSS-PL-$\ell_1$ and Inexact ITSS-PL-$\ell_2$, the parameter $\rho$ in the proximal term is set as $\rho=10^{-10}$, the parameter $\tau$ for defining the $\tau$-support $\rT^{k}$ is taken as $\tau=10^{-7}$.
In this experiment, the tolerance of the outer loop in the Inexact ITSS-PL-$\ell_{q}$ ($q=1,2$) is set to be $\epsilon=10^{-3}$, and the maximum outer iteration number is set to be $\text{MAXit}=25$.
We use $r_v=3\times10^5$  and $r_w = 200$ in the ADMM.
The tolerance of the inner loop is set to be $\epsilon_1 \text{ (or }\epsilon_2) = 10^{-5}$, and the maximum inner iteration number is set to be $\text{Maxit1 (or Maxit2)}=500$.

The initial value of the Inexact ITSS-PL-$\ell_q$ ($q=1,2$) is taken as the solution to the $\ell_q\text{aTV}$ model (i.e., the anisotropic $\ell_q\text{TV}^p$ model with $p=1$) solved by the ADMM described in Algorithm \ref{algo:ADMM-L1ITSSAPL} and Algorithm \ref{algo:ADMM-L2ITSSAPL} with $p=1$ and $\rho=0$.
For solving the	$\ell_q\text{aTV}$ model, we set $r_v=3\times10^3$, $r_w = 200$, $\epsilon_1 \text{ (or }\epsilon_2) = 10^{-5}$ and $\text{Maxit1 (or Maxit2)}=500$ in the ADMM.

\subsection{The convergence and $\tau$-support shrinkage properties}

In this subsection, we test the numerical evolution behavior of objective function values and support sizes when applying the Inexact ITSS-PL-$\ell_{q}$ in image restoration.
We test on three sample images corrupted by different blur kernels.
To reveal the convergence property, the maximum outer iteration number is set to be $\text{MAXit}=30$ in both algorithms.
We show the evolution behavior of  $\mathcal{F}(x^k)$ and the percentage $\sharp\rT^k/\sharp\rJ$ in Figure \ref{L1convergesupp} and Figure \ref{L2convergesupp}, for $q=1$ and 2 respectively.
It can be observed that $\mathcal{F}(x^k)$ is monotonically decreasing and converges, which reveals the global convergence of the proposed algorithms.
Meanwhile, the support size also decreases and converges, which is consistent with the support shrinkage property of Inexact ITSS-PL.

\begin{figure}[htbp]
	\hspace{15pt}
	{
		\begin{minipage}[t]{0.2\linewidth}
			\centering
			\centerline{\includegraphics[width=4.8cm]{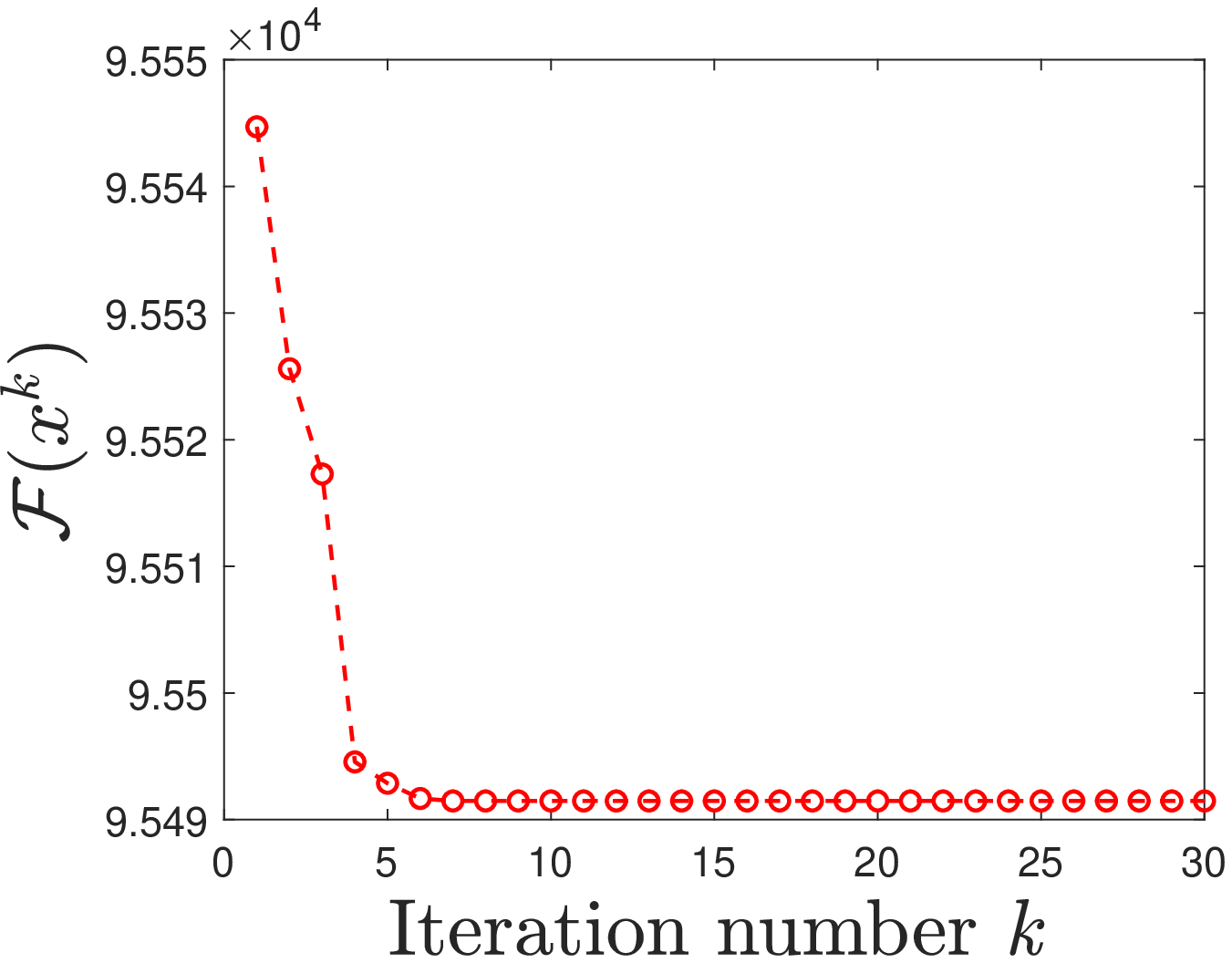}}
			\centerline{(1a) Squares} 
			\vspace{10pt}
			\centering
			\centerline{\includegraphics[width=4.8cm]{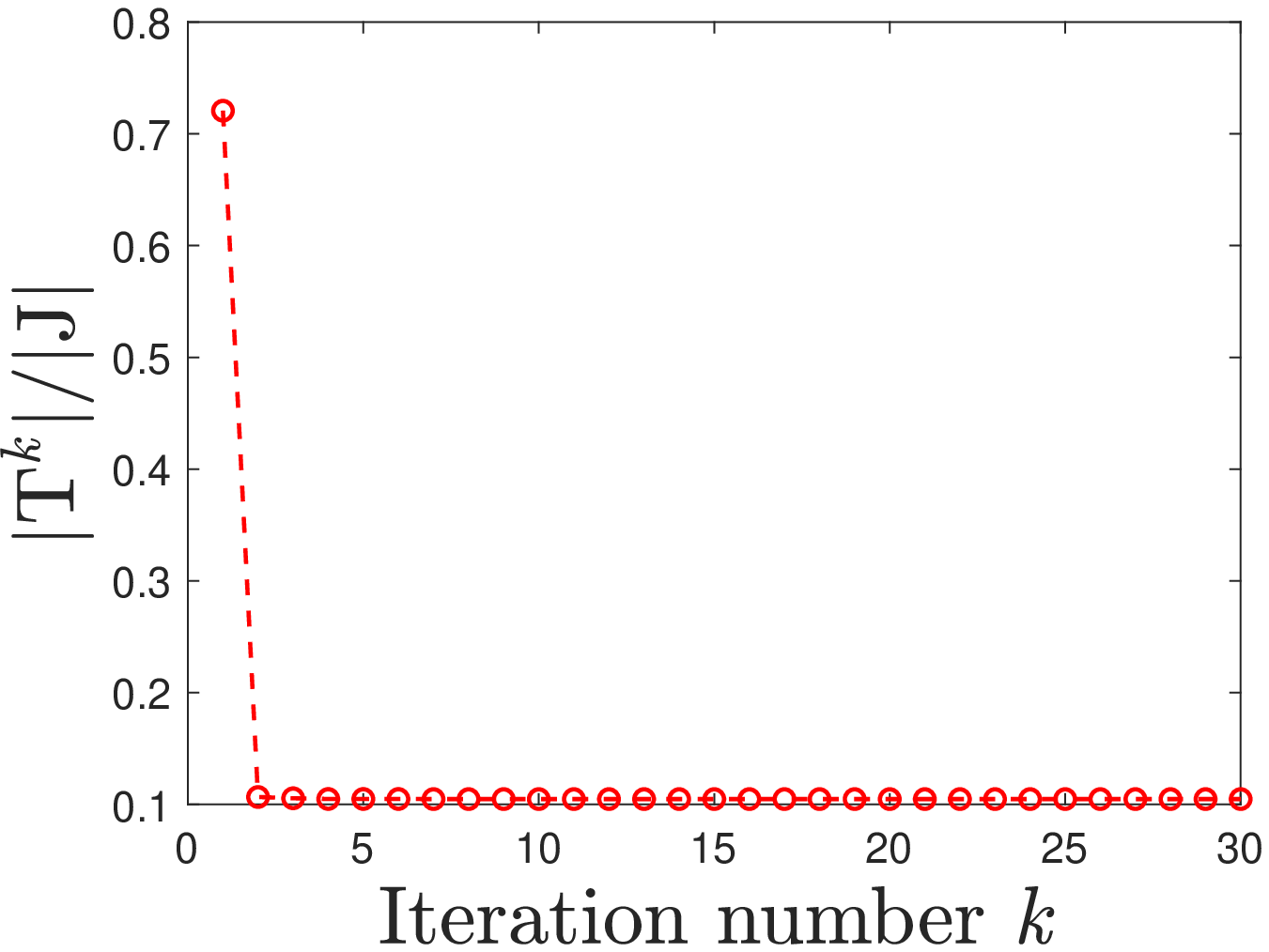}}
			\centerline{(1b) Squares} 
		\end{minipage}
	}
	\hspace{50pt}
	{
		\begin{minipage}[t]{0.2\linewidth}
			\centering
			\centerline{\includegraphics[width=4.8cm]{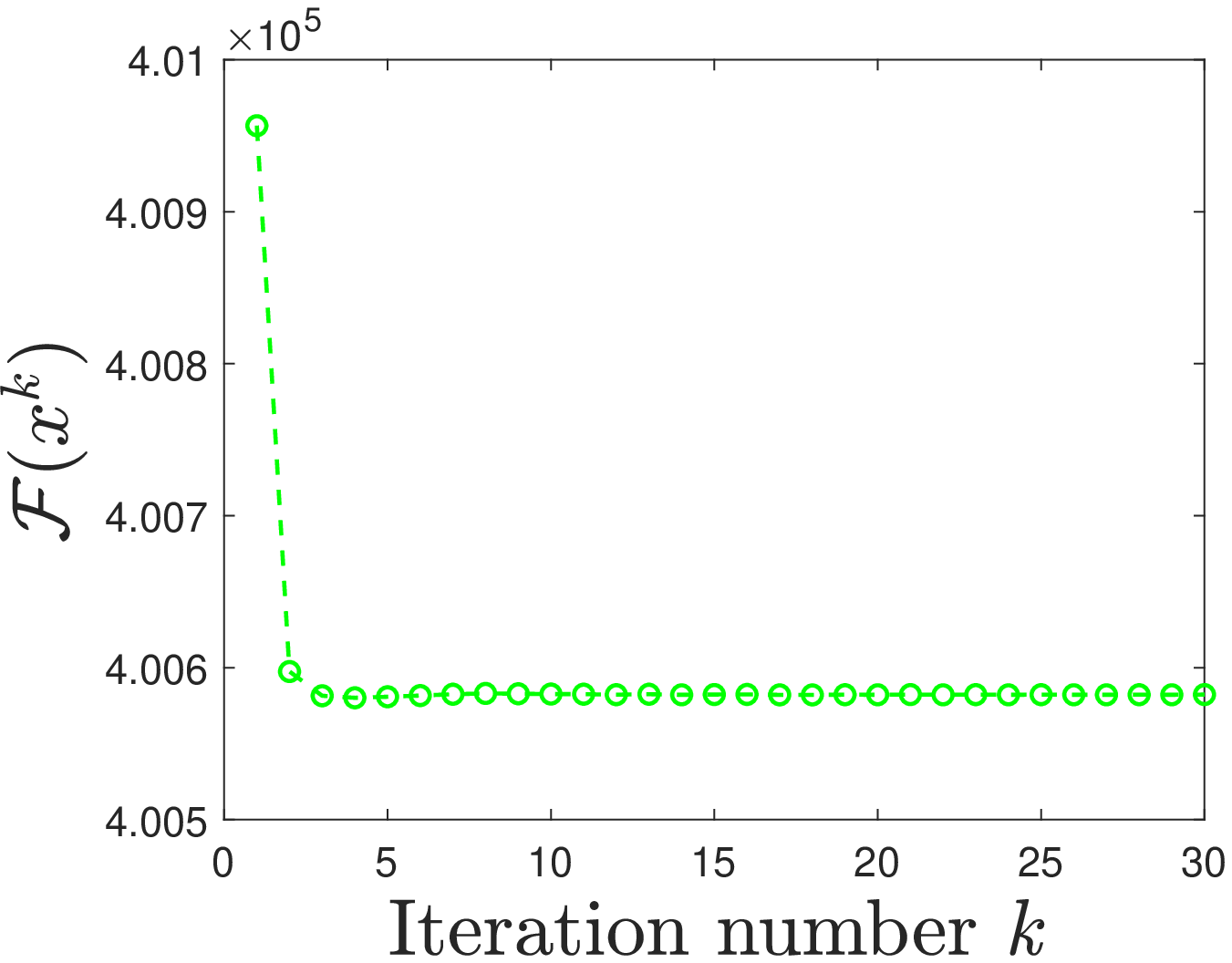}}
			\centerline{(2a) Twocircles} 
			\vspace{10pt}
			\centering
			\centerline{\includegraphics[width=4.8cm]{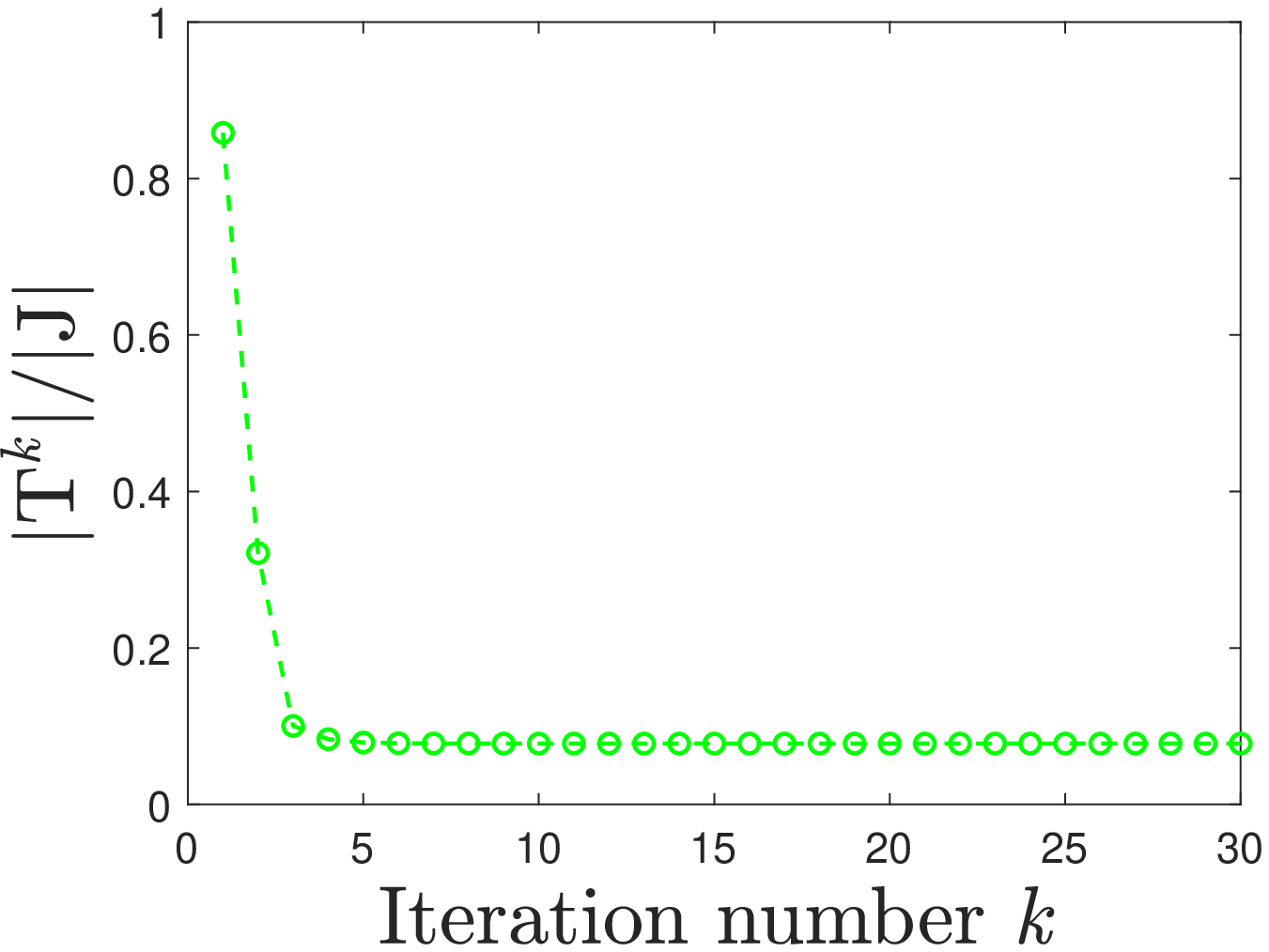}}
			\centerline{(2b) Twocircles} 
		\end{minipage}
	}
	\hspace{50pt}
	{
		\begin{minipage}[t]{0.2\linewidth}
			\centering
			\centerline{\includegraphics[width=4.8cm]{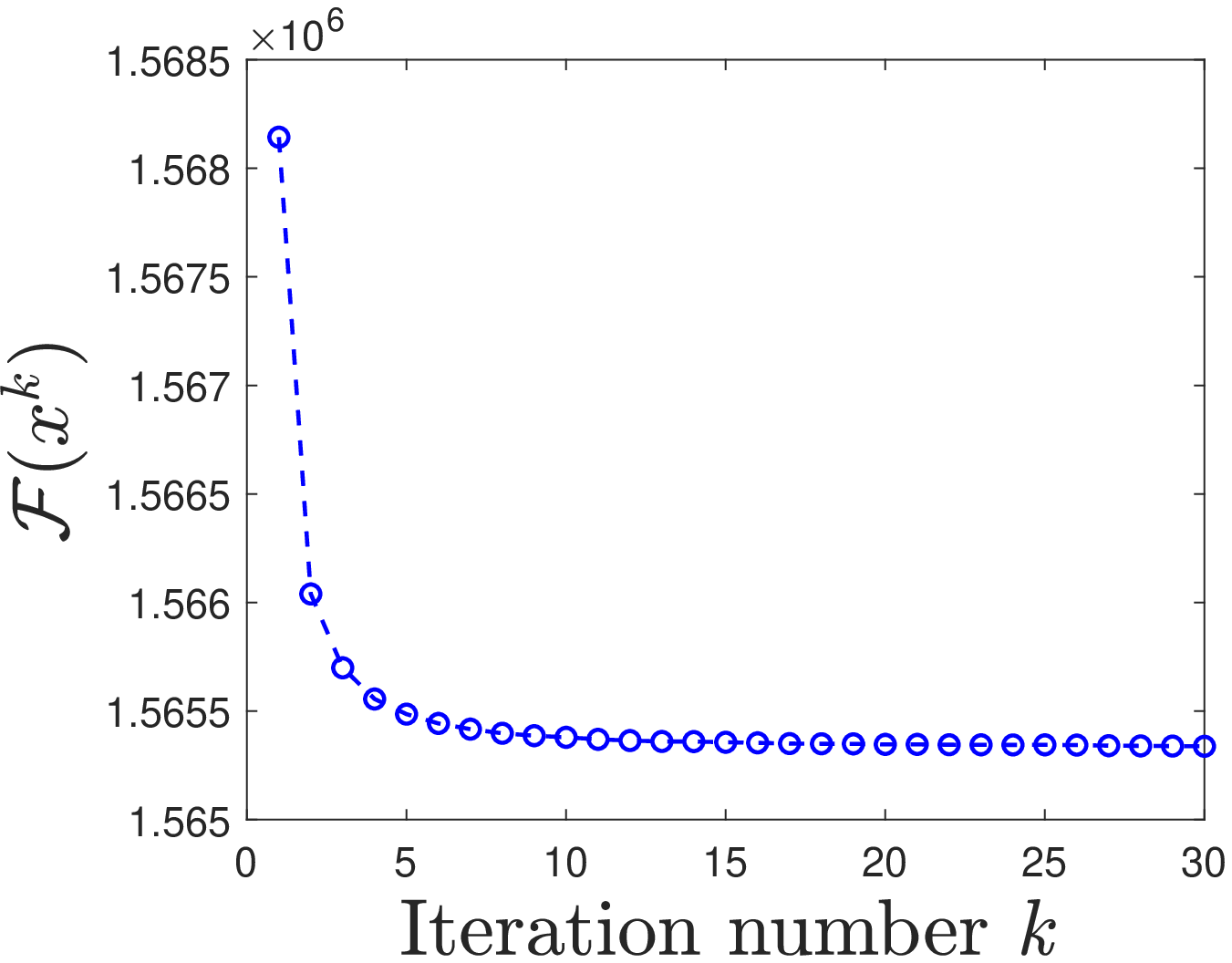}}
			\centerline{(3a) Cameraman} 
			\vspace{10pt}
			\centering
			\centerline{\includegraphics[width=4.8cm]{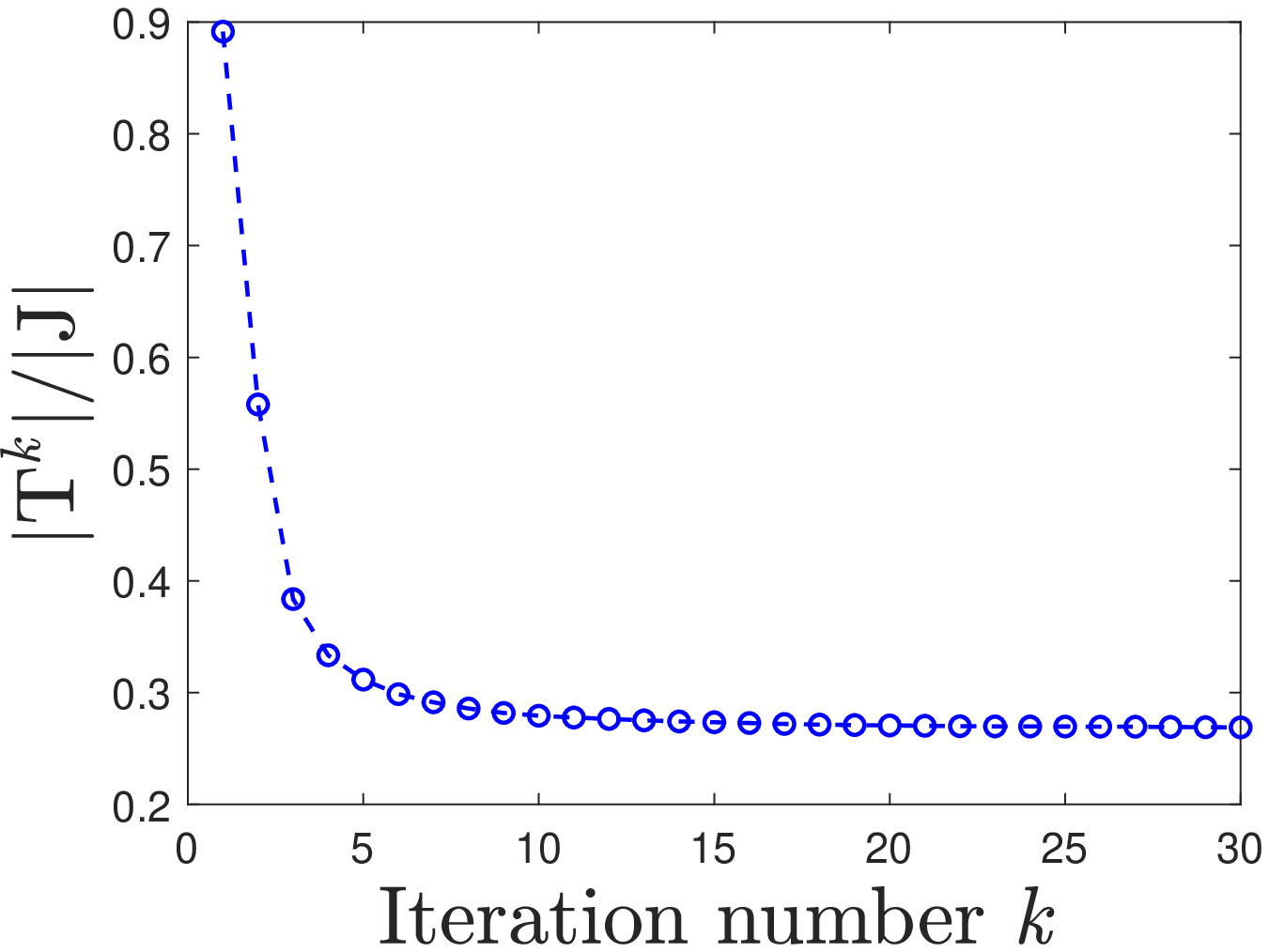}}
			\centerline{(3b) Cameraman} 
		\end{minipage}
	}
	\caption{
		The behavior of objective function values and support sizes of the Inexact ITSS-PL-$\ell_{1}$ for impulse noise removal.
			(1a)(1b): results of the "Squares" corrupted by the average blur; (2a)(2b): results of the "Twocircles" corrupted by the Gaussian blur; (3a)(3b): results of the "Cameraman" corrupted by the disk blur.
			(1a)(2a)(3a): the objective function value $\mathcal{F}(x^k)$ versus the outer iteration number $k$; 
			(1b)(2b)(3b): the percentage of cardinality of the support set $|\rT^k|\backslash|\rJ|$ versus the outer iteration number $k$. 
		}
	\label{L1convergesupp}
\end{figure}

\begin{figure}[htbp]
	\hspace{15pt}
	{
		\begin{minipage}[t]{0.2\linewidth}
			\centering
			\centerline{\includegraphics[width=4.8cm]{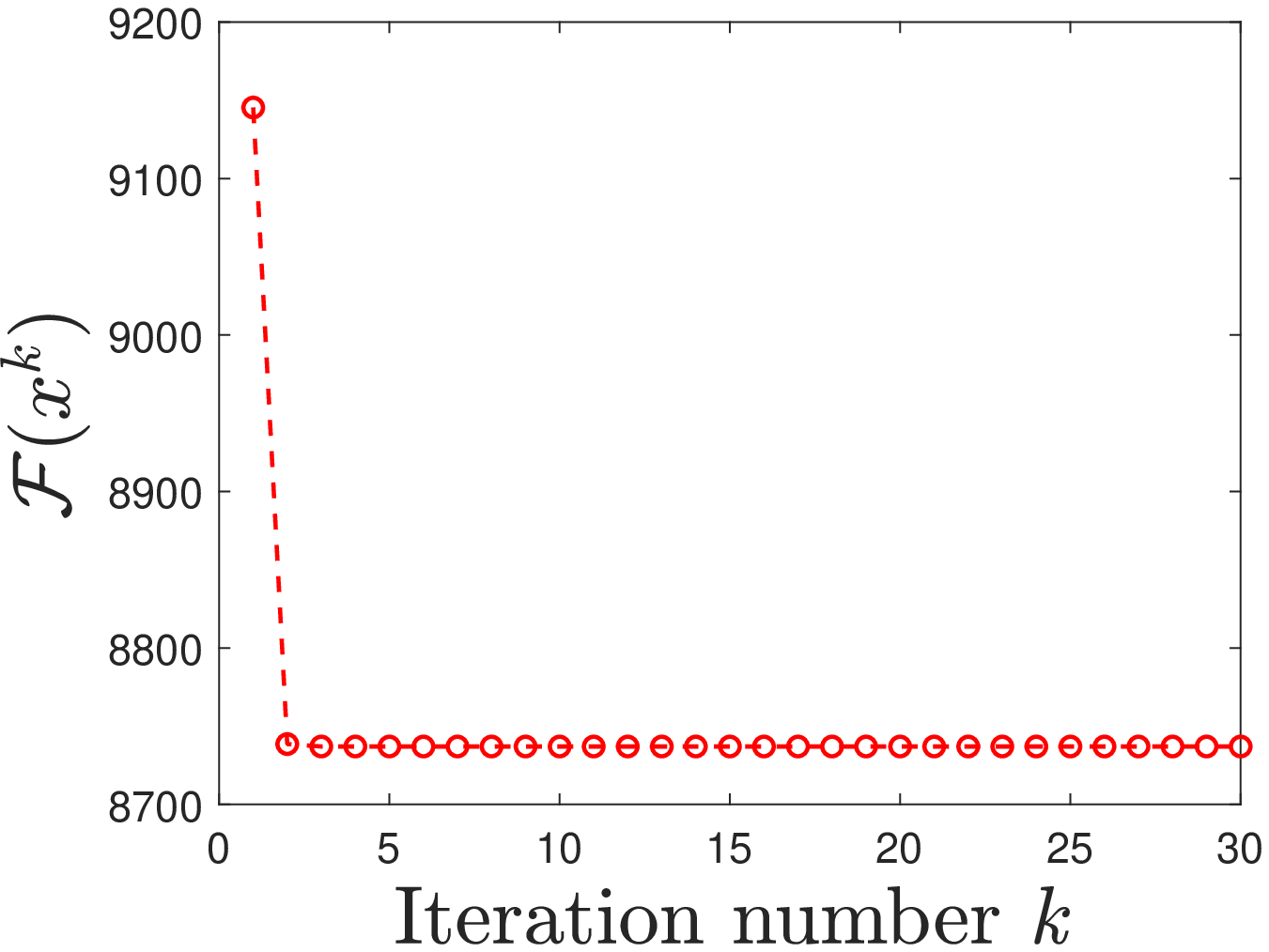}}
			\centerline{(1a) Squares}
			\vspace{10pt}
			\centering
			\centerline{\includegraphics[width=4.8cm]{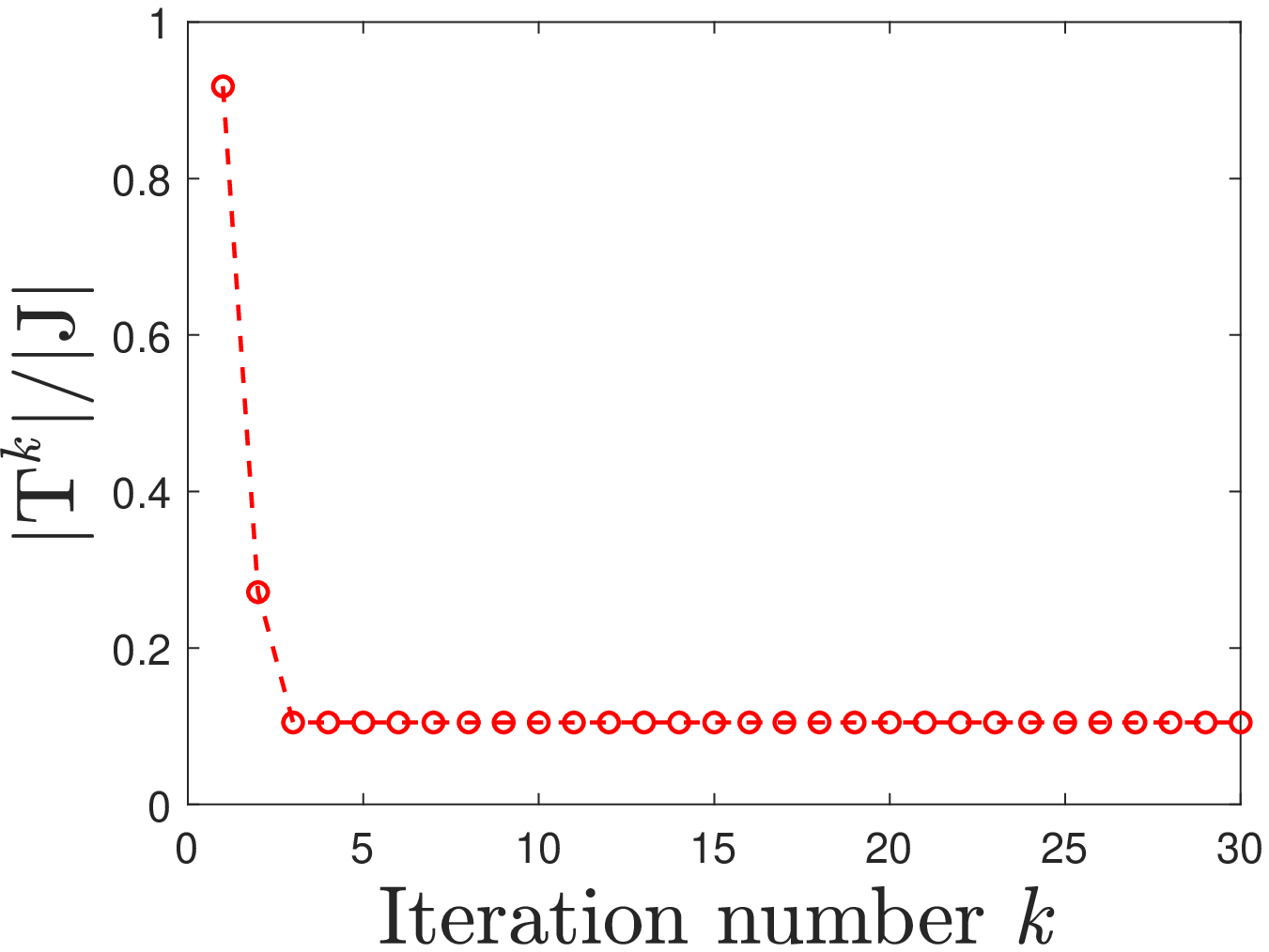}}
			\centerline{(1b) Squares}
		\end{minipage}
	}
	\hspace{50pt}
	{
		\begin{minipage}[t]{0.2\linewidth}
			\centering
			\centerline{\includegraphics[width=4.8cm]{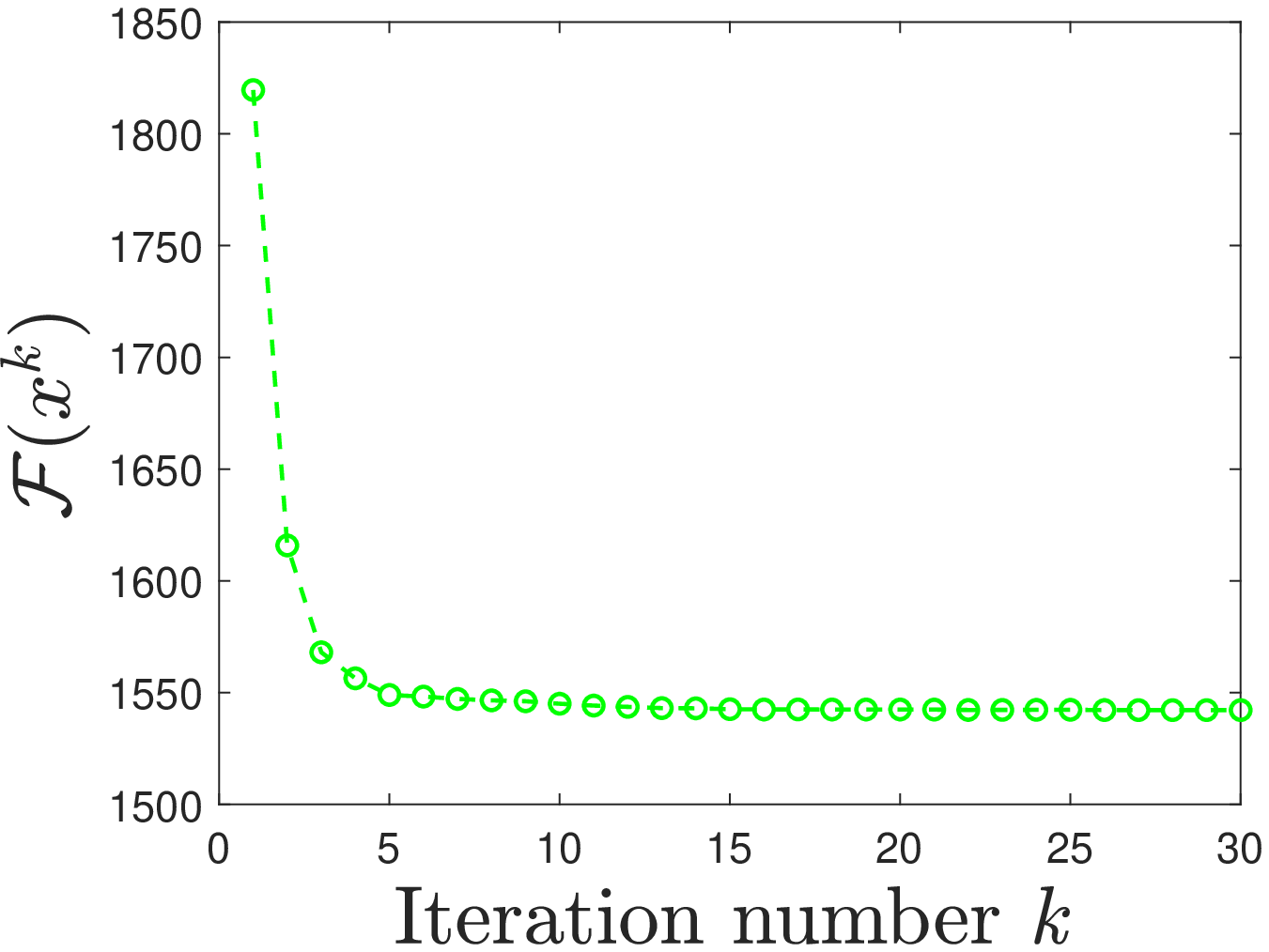}}
			\centerline{(2a) Twocircles}
			\vspace{10pt}
			\centering
			\centerline{\includegraphics[width=4.8cm]{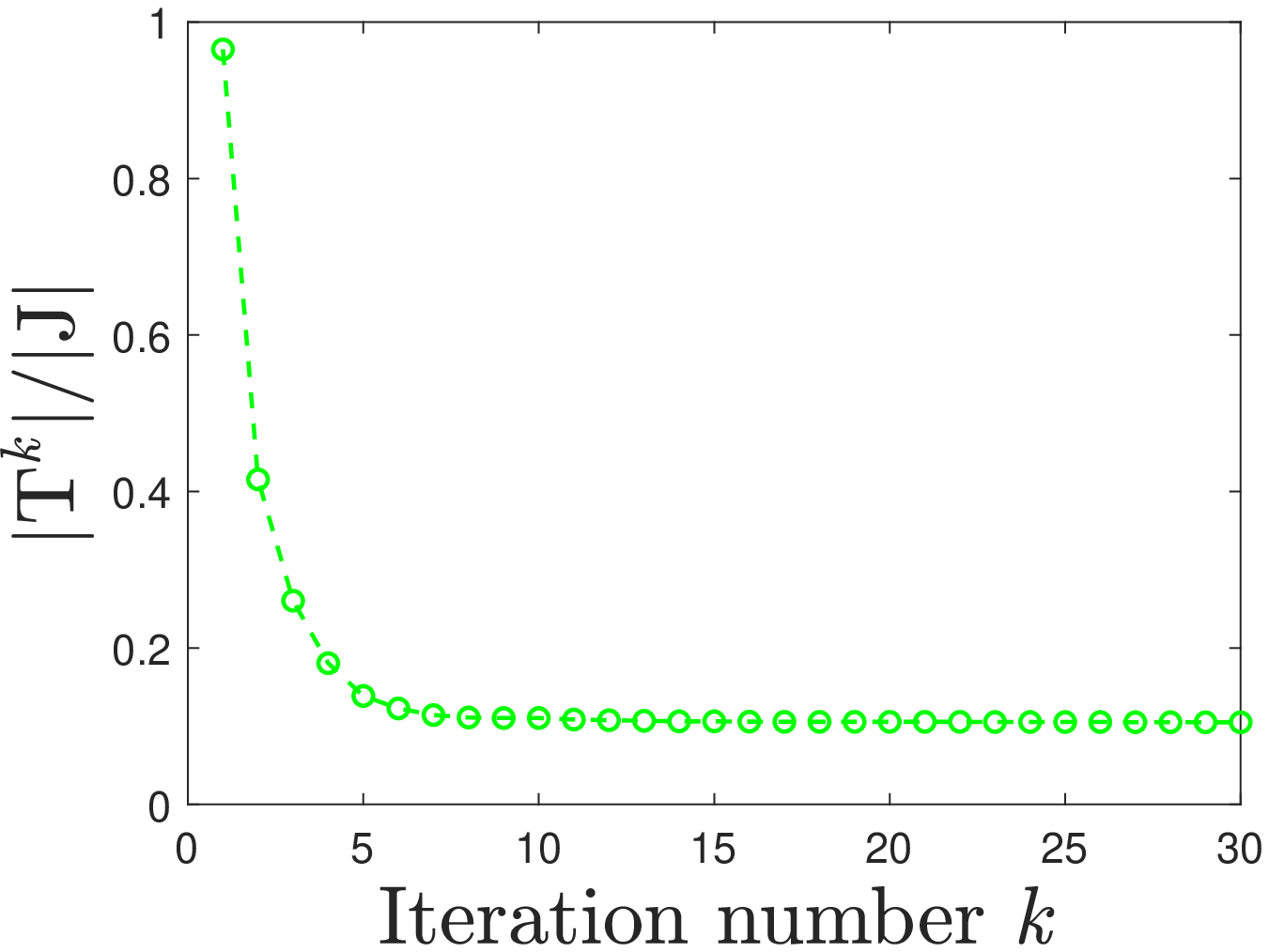}}
			\centerline{(2b) Twocircles}
		\end{minipage}
	}
	\hspace{50pt}
	{
		\begin{minipage}[t]{0.2\linewidth}
			\centering
			\centerline{\includegraphics[width=4.8cm]{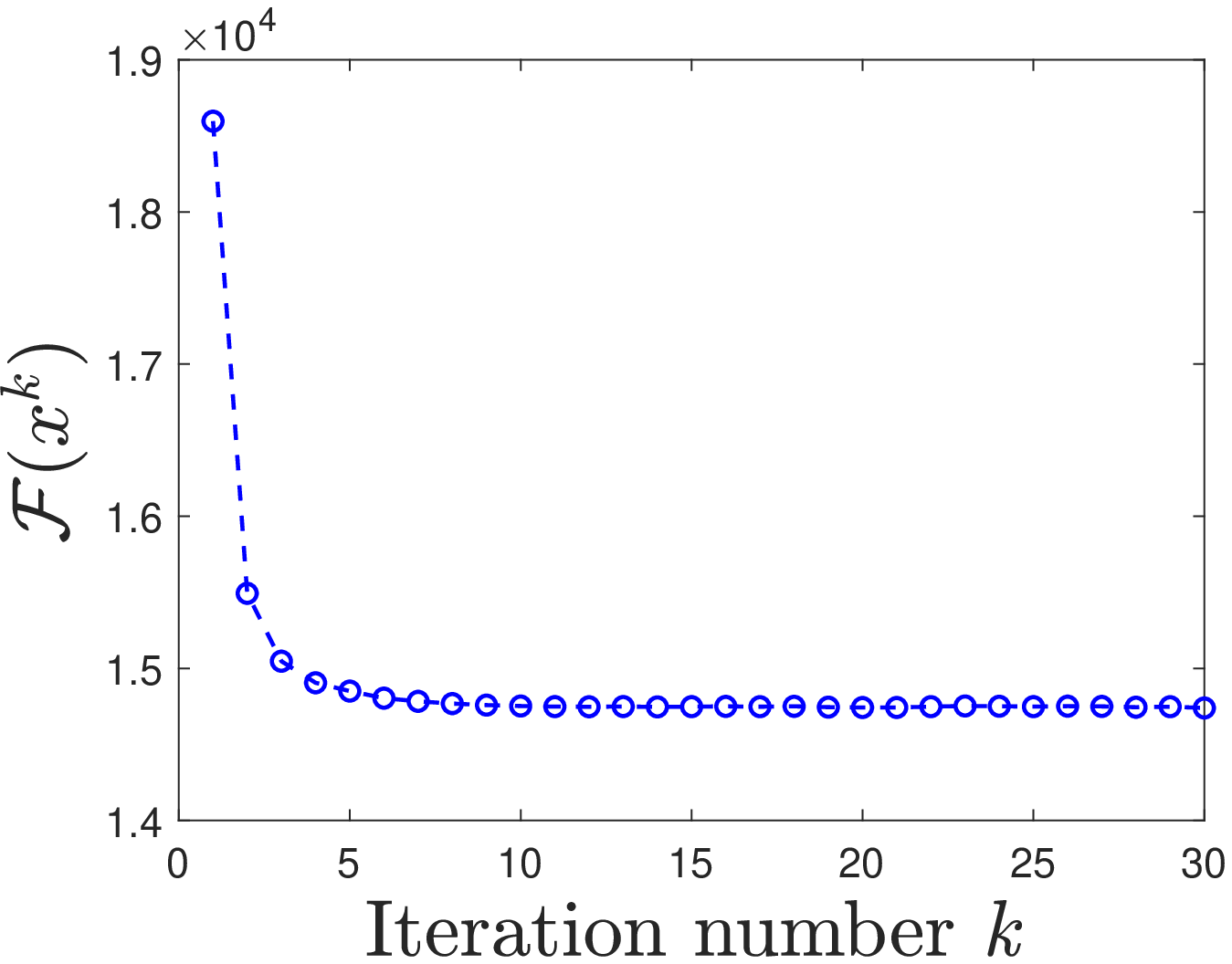}}
			\centerline{(3a) Cameraman}
			\vspace{10pt}
			\centering
			\centerline{\includegraphics[width=4.8cm]{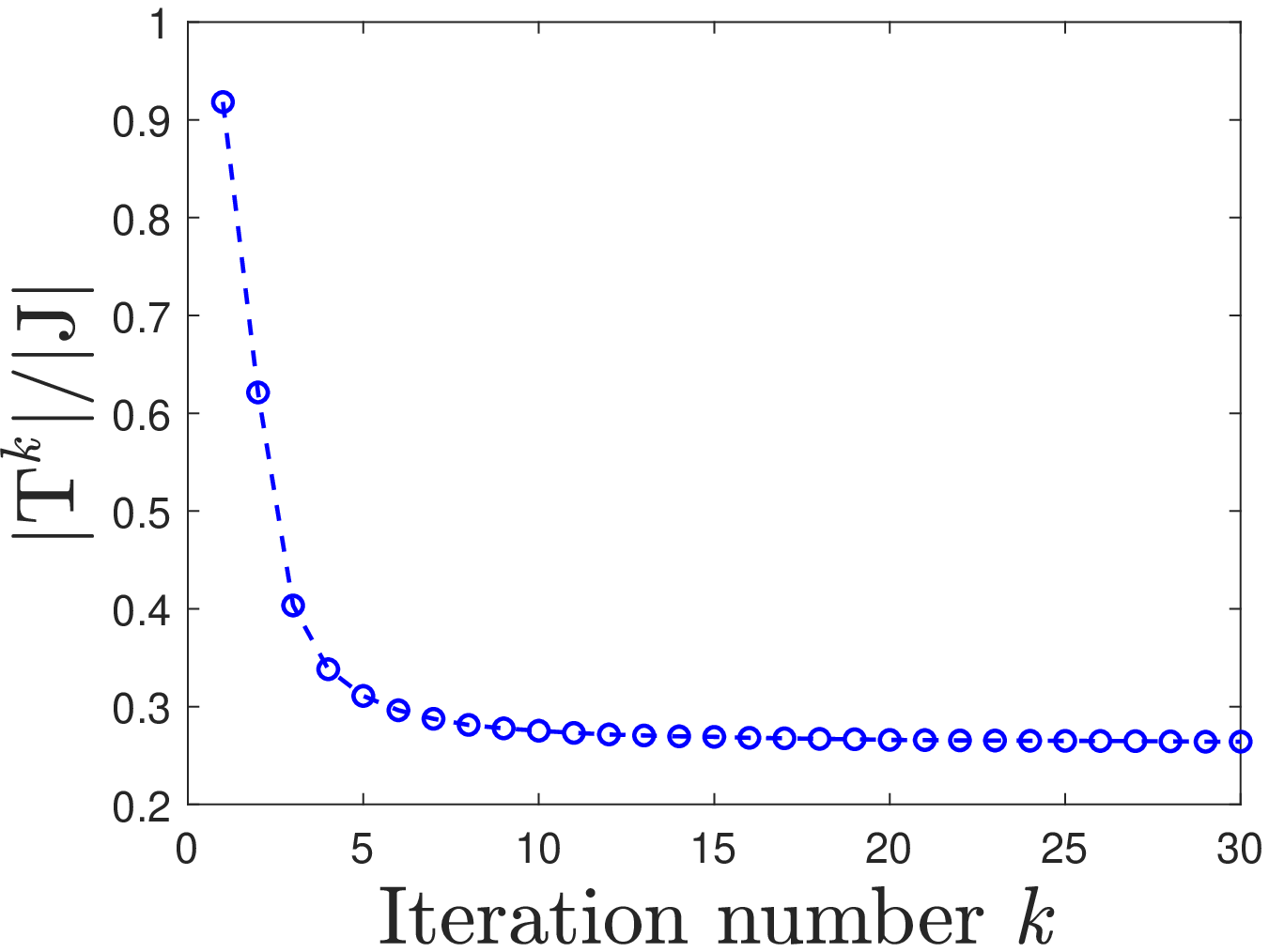}}
			\centerline{(3b) Cameraman}
		\end{minipage}
	}
	\caption{
		The behavior of objective function values and support sizes of the Inexact ITSS-PL-$\ell_{2}$ for Gaussian noise removal.
			(1a)(1b): results of the "Squares" corrupted by the average blur; (2a)(2b): results of the "Twocircles" corrupted by the Gaussian blur; (3a)(3b): results of the "Cameraman" corrupted by the disk blur.
			(1a)(2a)(3a): the objective function value $\mathcal{F}(x^k)$ versus the outer iteration number $k$; 
			(1b)(2b)(3b): the percentage of cardinality of the support set $|\rT^k|\backslash|\rJ|$ versus the outer iteration number $k$.	
	}
	\label{L2convergesupp}
\end{figure}

\subsection{Restoration performance}\label{sec:restoration}

In this subsection, we test the performance of $\ell_{q}\text{aTV}^p$ ($q=1,2$) model and the Inexact ITSS-PL in image restoration.
Figure \ref{restoration} shows four restoration results by our methods.
The performance of
the $\ell_{1}\text{aTV}^p$ and $\ell_{2}\text{aTV}^p$ are then compared to 
the $\ell_{1}\text{aTV}$ and $\ell_{2}\text{aTV}$, respectively.
The $\ell_{q}\text{aTV}(q=1,2)$ model for comparison is also solved by the ADMM with the same algorithm parameters used in the $\ell_{q}\text{aTV}$ model for initializing the $\ell_{q}\text{aTV}^p(q=1,2)$ model.
The PSNR values of the results recovered by different methods are illustrated in Table \ref{L1TVp} and Table \ref{L2TVp}.
It can be seen that in both cases of impulse noise and Gaussian noise, the $\ell_{q}\text{aTV}^p$ ($q=1,2$) model performs much better than the $\ell_{q}\text{aTV}$ ($q=1,2$) model on piecewise constant images, while the $\ell_{q}\text{aTV}$ model outperforms the $\ell_{}\text{aTV}^p$ model on the natural image.

\begin{table}[htbp]
	\centering
	\footnotesize{
	\begin{tabular}{|c|c|c|c|c|c|c|c|}
		\hline
		\multirow{2}{*}{Images}     & \multirow{2}{*}{} & \multicolumn{2}{c|}{average blur} & \multicolumn{2}{c|}{gaussian blur} & \multicolumn{2}{c|}{disk blur} \\ \cline{3-8}
		&     &$\ell_{1}\text{aTV}$   & $\ell_{1}\text{aTV}^p$   & $\ell_{1}\text{aTV}$    & $\ell_{1}\text{aTV}^p$   & $\ell_{1}\text{aTV}$  & $\ell_{1}\text{aTV}^p$    \\ \hline
		\multirow{2}{*}{Phantom}    & PSNR(dB)           & 39.89        & \textbf{44.07}          & 37.21         & \textbf{56.69}          & 99.56       & \textbf{156.96}       \\
		& $\beta$          & 25             & 12               & 500             & 140              & 100           & 25             \\
		\hline
		\multirow{2}{*}{Twocircles} & PSNR(dB)           & 28.96        & \textbf{35.37}          & 25.05         & \textbf{36.87}          & 63.51       & \textbf{216.01}       \\
		& $\beta$          & 35             & 30               & 700             & 650              & 160           & 75             \\
		\hline
		\multirow{2}{*}{Squares}        & PSNR(dB)           & 47.70        & \textbf{48.86}          & 91.70         & \textbf{150.23}         & 91.89       & \textbf{263.52}       \\
		& $\beta$          & 15             & 9                & 750             & 30               & 45            & 60             \\
		\hline
		\multirow{2}{*}{Text}       & PSNR(dB)           & 27.42        & \textbf{30.97}          & 24.29         & \textbf{28.34}          & 40.98       & \textbf{188.52}       \\
		& $\beta$          & 40             & 35               & 800             & 800              & 140           & 85             \\
		\hline
		\multirow{2}{*}{Cameraman}  & PSNR(dB)           & \textbf{29.51}        & 27.97          & \textbf{28.07}         & 27.74          & \textbf{30.04}       & 29.11        \\
		& $\beta$          & 30             & 25               & 550             & 700              & 130           & 160            \\
		\hline
	\end{tabular}
}
	\caption{
		Restoration from blurry images with $30\%$ salt-and-pepper impulse noise.}
	\label{L1TVp}
\end{table}

\begin{table}[htbp]
		\centering
		\footnotesize{
		\begin{tabular}{|c|c|c|c|c|c|c|c|}
			\hline
			\multirow{2}{*}{Images}     & \multirow{2}{*}{} & \multicolumn{2}{c|}{average blur} & \multicolumn{2}{c|}{gaussian blur} & \multicolumn{2}{c|}{disk blur}                                \\ \cline{3-8}
			&  & $\ell_{2}\text{aTV}$   & $\ell_{2}\text{aTV}^p$   & $\ell_{2}\text{aTV}$    & $\ell_{2}\text{aTV}^p$   & $\ell_{2}\text{aTV}$   & $\ell_{2}\text{aTV}^p$          \\ \hline
			\multirow{2}{*}{Phantom}    & PSNR(dB)           & 56.89        & \textbf{73.76}          & 28.44        & \textbf{34.57}          & 45.21                       & \textbf{51.99}                       \\
			& $\beta$       & $7\times10^3$    & $2\times10^4$    & $3\times10^5$    & $4\times10^5$    & $4\times10^4$    & $5\times10^4$ \\
			\hline
			\multirow{2}{*}{Twocircles} & PSNR(dB)           & 46.79        & \textbf{69.72}          & 20.34        & \textbf{23.47}          & 36.39                       & \textbf{45.34}                       \\
			& $\beta$       & $2\times10^4$    & $2\times10^4$    & $6\times10^5$    & $5\times10^5$    & $8\times10^4$    & $9\times10^4$ \\
			\hline
			\multirow{2}{*}{Squares}        & PSNR(dB)           & 57.67        & \textbf{71.83}          & 43.47         & \textbf{53.03}          & 42.08                       & \textbf{60.83}                       \\
			& $\beta$       & $8\times10^3$    & $7\times10^3$    & $5\times10^4$    & $6\times10^4$    & $4\times10^4$    & $7\times10^4$ \\
			\hline
			\multirow{2}{*}{Text}       & PSNR(dB)           & 49.00        & \textbf{70.60}          & 21.75         & \textbf{23.42}          & 39.63                       & \textbf{61.41}                       \\
			& $\beta$       & $2\times10^4$    & $2\times10^4$    & $5\times10^5$    & $7\times10^5$    & $8\times10^4$    & $3\times10^4$ \\
			\hline
			\multirow{2}{*}{Cameraman}  & PSNR(dB)           & \textbf{35.48}        & 34.80          & \textbf{25.63 }        & 25.11          & \textbf{31.17}                       & 30.60                       \\
			& $\beta$       & $8\times10^4$    & $2\times10^5$    & $2\times10^5$    & $3\times10^5$    & $1\times10^5$    & $2\times10^5$ \\
			\hline
		\end{tabular}
}
	\caption{
		Restoration from blurry images with 0 mean and $10^{-6}$ variance Gaussian noise.}
	\label{L2TVp}
\end{table}

\begin{figure}[htbp]
	\centering
	{
		\begin{minipage}[t]{0.2\linewidth}
			\centering
			\centerline{\includegraphics[width=3.2cm]{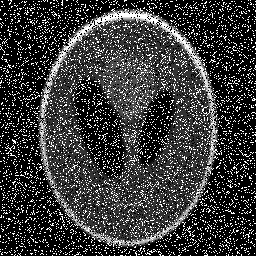}}
			\centerline{(1a) average blur}
			\vspace{10pt}
			\centering
			\centerline{\includegraphics[width=3.2cm]{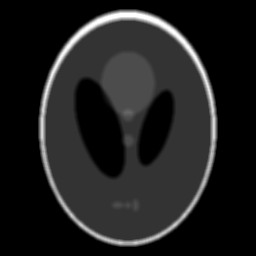}}
			\centerline{(2a) average blur}
		\end{minipage}
	}
	\hspace{5pt}
	{
		\begin{minipage}[t]{0.2\linewidth}
			\centering
			\centerline{\includegraphics[width=3.2cm]{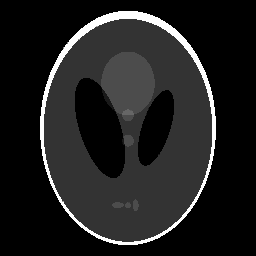}}
			\centerline{(1b) recovery of (1a)}
			\vspace{10pt}
			\centering
			\centerline{\includegraphics[width=3.2cm]{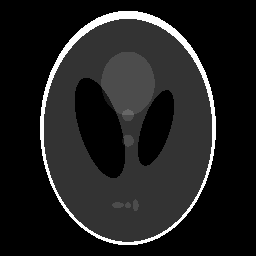}}
			\centerline{(2b) recovery of (2a)}
		\end{minipage}
	}
	\hspace{25pt}
	{
		\begin{minipage}[t]{0.2\linewidth}
			\centering
			\centerline{\includegraphics[width=3.2cm]{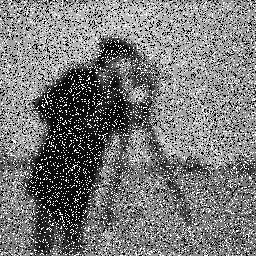}}
			\centerline{(1c) disk blur}
			\vspace{10pt}
			\centering
			\centerline{\includegraphics[width=3.2cm]{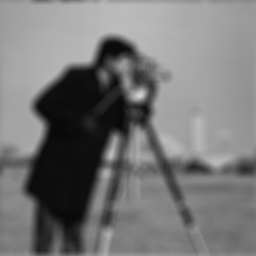}}
			\centerline{(2c) disk blur}
		\end{minipage}
	}
	\hspace{5pt}
	{
		\begin{minipage}[t]{0.2\linewidth}
			\centering
			\centerline{\includegraphics[width=3.2cm]{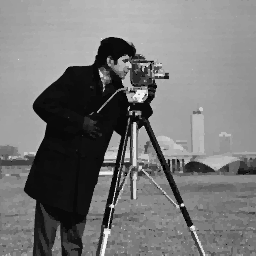}}
			\centerline{(1d) recovery of (1c)}
			\vspace{10pt}
			\centering
			\centerline{\includegraphics[width=3.2cm]{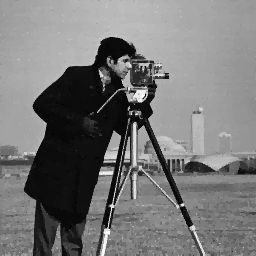}}
			\centerline{(2d) recovery of (2c)}
		\end{minipage}
	}
	\caption{
			Image restoration by ITSS-PL-$\ell_1$ and ITSS-PL-$\ell_2$. (1a)(1c): The test images degraded by different types of blur kernels and the salt-and-pepper impulse noise with noise level 30$\%$. (2a)(2c): The test images degraded by different types of blur kernels and the i.i.d. Gaussian white noise with variance $10^{-6}$. (1b)(1d): The recovered images by the $\ell_{1}\text{aTV}^{p}$ model. (2b)(2d): The recovered images by the $\ell_{2}\text{aTV}^{p}$ model.
	}
	\label{restoration}
\end{figure}

\subsection{Applications in Image Segmentation}

Our proposed algorithm can also be applied in the two-stage image segmentation method.
Given a blurry and noisy image, the segmentation problem is to partition the image into several regions based on the image intensity.
The two-stage image segmentation method (\cite{cai2013segmentation}) includes two major steps: Stage One is to get a piecewise constant approximation of the observation; Stage Two is a simple thresholding operation applied to the approximation  
and to obtain the segmentation result.
We apply the $\ell_q\text{aTV}^p$ model, and the $\ell_q\text{aTV}$ model as a comparison, at the first stage of finding the approximations.
At the second stage, we use the same segmentation method proposed in \cite{cai2013segmentation}.

In the experiment of image segmentation, we use two piece-wise constant test images shown in Figure \ref{segorigin}, which ground truth of segmentation can be obtained by MATLAB function \textit{tabulate}, and thus convenient for the quantitive comparision defined later.
We use three different blurring kernels: (1) average blur  (fspecial('average',9)); 
(2) Gaussian blur  
(fspecial('gaussian',[15,15],10)); 
(3) disk blur (fspecial('disk',6)).
Two types of noises are considered:
(1) $40\%$ salt-and-pepper impulse noise; (2) Gaussian noise with mean 0 variance $10^{-6}$.
At the first stage of image restoration, 
the $\ell_1\text{aTV}^p$ model with $p=0.5$ and the $\ell_1\text{aTV}$ model, as a comparison, are used to obtain a piecewise constant approximation of the observation degraded by the impulse noise; the $\ell_2\text{aTV}^p$ model with $p=0.5$ and the $\ell_2\text{aTV}$ model are used to get an approximation of the images degraded by the Gaussian noise.
All the algorithm parameters for solving the $\ell_q\text{aTV}^p$ model and the $\ell_q\text{aTV}$ model ($q=1,2$) are the same to those used in the experiments of image deconvolution in Section \ref{subsec:parameters}.
The $Jaccard$ $Similarity$ (denoted as JS, \cite{jaccard1912the}) value is used to evaluate the segmentation results, which is defined as follows:
\begin{equation}
\text{JS}(S_{\text{GT}}^{\text{phase}},S_{\text{Alg}}^{\text{phase}}):=\frac{|S_{\text{GT}}^{\text{phase}}\cap S_{\text{Alg}}^{\text{phase}}|}{|S_{\text{GT}}^{\text{phase}}\cup S_{\text{Alg}}^{\text{phase}}|},
\end{equation}
where $S_{\text{GT}}^{\text{phase}}$ and $S_{\text{Alg}}^{\text{phase}}$ denote the region of one certain phase in the ground truth and the region in the segmentation result, $|\cdot|$ denotes the area of a region. Clearly,
the higher the JS value, the better the corresponding segmentation result.

\begin{figure}[htbp]
	\centering
	\subfigure[Geometry]
	{
		\includegraphics[width=3.5cm]{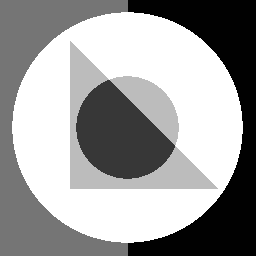}
	}
	\quad
	\subfigure[Phantom]
	{
		\includegraphics[width=3.5cm]{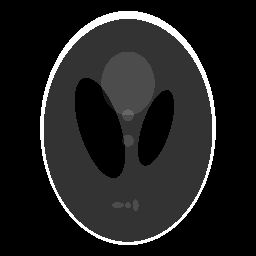}
	}
	\quad
	\caption{
		Test images for segmentation. (a): 5-phase Geometry ($256\times 256$); (b): 6-phase Phantom ($256\times 256$).}
	\label{segorigin}
\end{figure}

For the case of impulse noise and Gaussian noise, the segmentation results 
are shown in Table \ref{L1JSvalue} and Table \ref{L2JSvalue}, respectively. For a visual illustration, one may check Figure \ref{L1seg} and Figure \ref{L2seg}.
The results in Table \ref{L1JSvalue} and Table \ref{L2JSvalue} indicate that the performances of our approach is better than those of the $\ell_{q}\text{aTV}$ ($q=1,2$) model.
In particular, our method performs a little bit better on phases with large area than the $\ell_q\text{aTV}$ model does, and much more better on phases of small areas, e.g., phase 5 in "Geometry" and phase 5, 6 in "Phantom".
	This can also be observed from the sample regions within the yellow boxes shown in (2b) (2c) and (2d) of Figure \ref{L1seg} or Figure \ref{L2seg}, which reveals the advantage of the $\ell_q\text{aTV}^p$ model on dealing with the adjacent phases of small areas  and similar intensities.

\begin{figure}[htbp]
	\centering
	\begin{minipage}[t]{0.18\linewidth}
		\centering
		\centerline{\includegraphics[width=3.2cm]{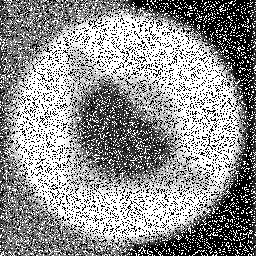}}
		\centerline{(1a) corrupted image}
		\vspace{10pt}
		\centering
		\centerline{\includegraphics[width=3.2cm]{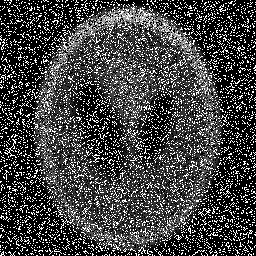}}
		\centerline{(2a) corrupted image}
	\end{minipage}
	\hspace{20pt}
	\begin{minipage}[t]{0.18\linewidth}
		\centering
		\centerline{\includegraphics[width=3.2cm]{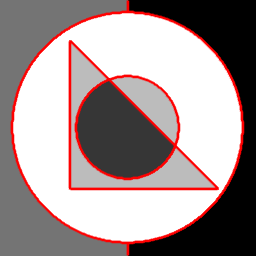}}
		\centerline{(1b) ground truth}
		\vspace{10pt}
		\centering
		\centerline{\includegraphics[width=3.2cm]{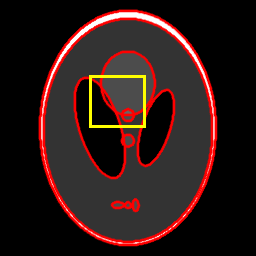}}
		\centerline{(2b) ground truth}
	\end{minipage}
	\hspace{20pt}
	\begin{minipage}[t]{0.18\linewidth}
		\centering
		\centerline{\includegraphics[width=3.2cm]{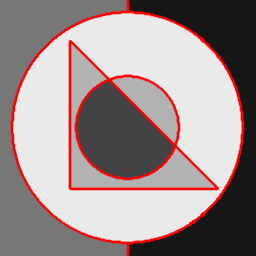}}
		\centerline{(1c) $\ell_{1}\text{aTV}$}
		\vspace{10pt}
		\centering
		\centerline{\includegraphics[width=3.2cm]{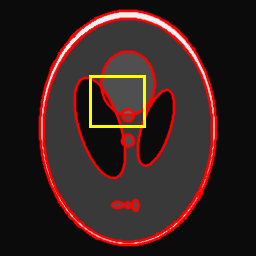}}
		\centerline{(2c) $\ell_{1}\text{aTV}$}
	\end{minipage}
	\hspace{20pt}
	\begin{minipage}[t]{0.18\linewidth}
		\centering
		\centerline{\includegraphics[width=3.2cm]{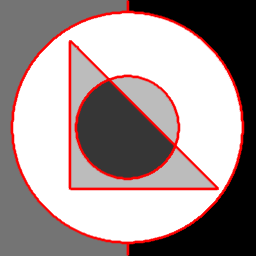}}
		\centerline{(1d) $\ell_{1}\text{aTV}^{p}$}
		\vspace{10pt}
		\centering
		\centerline{\includegraphics[width=3.2cm]{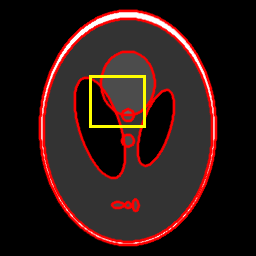}}
		\centerline{(2d) $\ell_{1}\text{aTV}^{p}$}
	\end{minipage}
	\caption{
			(1a) (2a): images corrupted by Gaussian blur and salt-and-pepper noise with level 40$\%$;
			(1b) to (1d) and (2b) to (2d): segmentation results of groud truth and various methods.
	}
	\label{L1seg}
\end{figure}

\begin{table}[htbp]
	\centering
	\footnotesize{
		\begin{tabular}{|c|c|c|c|c|c|c|c|c|}
			\hline
			\multirow{2}{*}{Images}    & \multicolumn{2}{c|}{\multirow{2}{*}{}}  & \multicolumn{2}{c|}{average blur} & \multicolumn{2}{c|}{gaussian blur} & \multicolumn{2}{c|}{disk blur} \\ \cline{4-9}
			& \multicolumn{2}{c|}{}  & $\ell_{1}\text{aTV}$  & $\ell_{1}\text{aTV}^p$  & $\ell_{1}\text{aTV}$  & $\ell_{1}\text{aTV}^{p}$ & $\ell_{1}\text{aTV}$  & $\ell_{1}\text{aTV}^p$  \\ \hline
			\multirow{6}{*}{Geometry}  & \multirow{5}{*}{JS} & phase 1(43.667\%) & 0.99951       & \textbf{0.99990}         & 0.99969        & \textbf{0.99997}         & \textbf{1.00000}      & \textbf{1.00000}      \\ \cline{3-9}
			&                     & phase 2(18.370\%) & 0.99859       & \textbf{0.99983}        & 0.99875        & \textbf{1.00000}         & \textbf{1.00000}      & \textbf{1.00000}      \\ \cline{3-9}
			&                     & phase 3(18.018\%) & \textbf{0.99958}       & \textbf{0.99958}        & 0.99983        & \textbf{1.00000}        & \textbf{1.00000}    & \textbf{1.00000}       \\ \cline{3-9}
			&                     & phase 4(10.487\%) & 0.99811       & \textbf{0.99971}        & 0.99855        & \textbf{0.99986}        & \textbf{1.00000}      & \textbf{1.00000}       \\ \cline{3-9}
			&                     & phase 5( 9.459\%) & 0.99759       & \textbf{1.00000}        & 0.99743        & \textbf{1.00000}         & \textbf{1.00000}      & \textbf{1.00000}       \\ \cline{2-9}
			& \multicolumn{2}{c|}{$\beta$}             & 55             & 65               & 250             & 350              & 110           & 45             \\
			\hline
			\multirow{7}{*}{Phantom}   & \multirow{6}{*}{JS} & phase 1(58.177\%) & 0.99874       & \textbf{0.99997}         & 0.99782        & \textbf{1.00000}        & 0.99990      & \textbf{1.00000}      \\ \cline{3-9}
			&                     & phase 2(32.927\%) & 0.99722       & \textbf{0.99995}         & 0.99579        & \textbf{1.00000}         & 0.99944      & \textbf{1.00000}       \\ \cline{3-9}
			&                     & phase 3( 4.343\%) & 0.99789       & \textbf{1.00000}         & 0.99824        & \textbf{1.00000}         & \textbf{1.00000}      & \textbf{1.00000}       \\ \cline{3-9}
			&                     & phase 4( 4.335\%) & 0.98577       & \textbf{1.00000}         & 0.97965        & \textbf{1.00000}         & 0.99649      & \textbf{1.00000}       \\ \cline{3-9}
			&                     & phase 5( 0.139\%) & 0.68421       & \textbf{0.97850}         & 0.61074        & \textbf{1.00000}         & 0.97850      & \textbf{1.00000}       \\ \cline{3-9}
			&                     & phase 6( 0.079\%) & 0.75362       & \textbf{1.00000}         & 0.72222        & \textbf{1.00000}         & \textbf{1.00000}      & \textbf{1.00000}       \\ \cline{2-9}
			& \multicolumn{2}{c|}{$\beta$}             & 80            & 40               & 190             & 70              & 110           & 40             \\
			\hline
		\end{tabular}
	}
	\caption{
		JS values of segmentation results for Geometry and Phantom corrupted by $40\%$ impulse noise. The percentages in the brackets are the ratios of the phase areas to the whole image in the ground truth.}
	\label{L1JSvalue}
\end{table}

\begin{figure}[htbp]
	\centering
	\begin{minipage}[t]{0.18\columnwidth}
		\centering
		\centerline{\includegraphics[width=3.2cm]{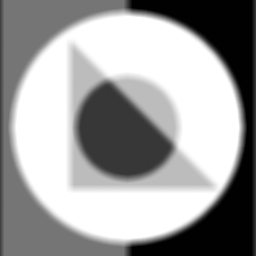}}
		\centerline{(1a) corrupted image}
		\vspace{10pt}
		\centering
		\centerline{\includegraphics[width=3.2cm]{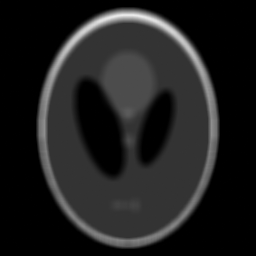}}
		\centerline{(2a) corrupted image}
	\end{minipage}
	\hspace{20pt}
	\begin{minipage}[t]{0.18\columnwidth}
		\centering
		\centerline{\includegraphics[width=3.2cm]{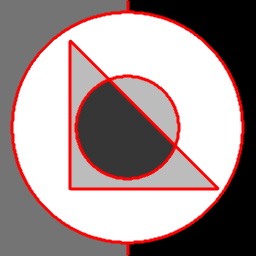}}
		\centerline{(1b) ground truth}
		\vspace{10pt}
		\centering		
		\centerline{\includegraphics[width=3.2cm]{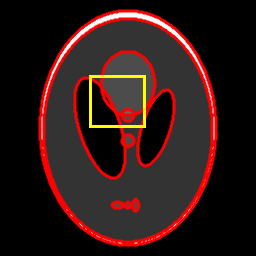}}
		\centerline{(2b) ground truth}
	\end{minipage}
	\hspace{20pt}
	\begin{minipage}[t]{0.18\columnwidth}
		\centering
		\centerline{\includegraphics[width=3.2cm]{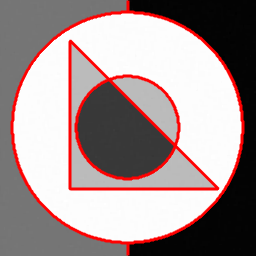}}
		\centerline{(1c) $\ell_{2}\text{aTV}$}
		\vspace{10pt}
		\centering
		\centerline{\includegraphics[width=3.2cm]{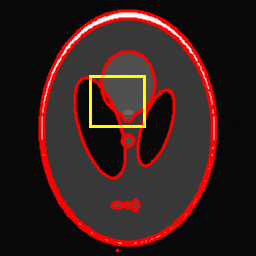}}
		\centerline{(2c) $\ell_{2}\text{aTV}$}
	\end{minipage}
	\hspace{20pt}
	\begin{minipage}[t]{0.18\columnwidth}
		\centering
		\centerline{\includegraphics[width=3.2cm]{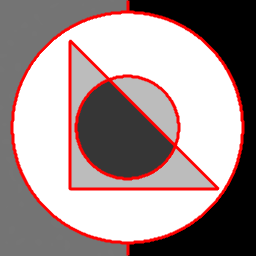}}
		\centerline{(1d) $\ell_{2}\text{aTV}^{p}$}
		\vspace{10pt}
		\centering
		\centerline{\includegraphics[width=3.2cm]{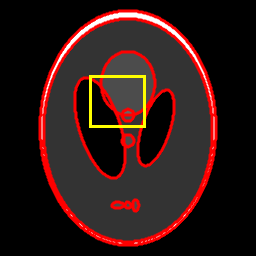}}
		\centerline{(2d) $\ell_{2}\text{aTV}^{p}$}
	\end{minipage}
	\caption{
			(1a)(2a): images corrupted by average blur and Gaussian noise with $0$ mean and $10^{-6}$ variance;
			(1b) to (1d) and (2b) to (2d): segmentation results of ground truth and various methods.
	}
	\label{L2seg}
\end{figure}

\begin{table}[htbp]
	\centering
	\footnotesize{
		\begin{tabular}{|c|c|c|c|c|c|c|c|c|}
			\hline
			\multirow{2}{*}{Images}    & \multicolumn{2}{c|}{\multirow{2}{*}{}}  & \multicolumn{2}{c|}{average blur} & \multicolumn{2}{c|}{gaussian blur} & \multicolumn{2}{c|}{disk blur} \\ \cline{4-9}
			& \multicolumn{2}{c|}{}  & $\ell_{2}\text{aTV}$  & $\ell_{2}\text{aTV}^p$   & $\ell_{2}\text{aTV}$  & $\ell_{2}\text{aTV}^p$  & $\ell_{2}\text{aTV}$  & $\ell_{2}\text{aTV}^p$    \\ \hline
			\multirow{6}{*}{Geometry}  & \multirow{5}{*}{JS} & phase 1(43.666\%) & \textbf{1.00000}       & \textbf{1.00000}        & 0.99710        & \textbf{0.99951}         & 0.99990      & \textbf{0.99993}       \\ \cline{3-9}
			&                     & phase 2(18.370\%) & \textbf{1.00000}       & \textbf{1.00000}         & 0.99552        & \textbf{1.00000}         & 0.99975      & \textbf{1.00000}       \\ \cline{3-9}
			&                     & phase 3(18.018\%) & \textbf{1.00000}       & \textbf{1.00000}         & 0.99460        & \textbf{1.00000}         & 0.99975      & \textbf{0.99983}       \\ \cline{3-9}
			&                     & phase 4(10.487\%) & \textbf{1.00000}       & \textbf{1.00000}         & 0.98349        & \textbf{0.99797}         & 0.99942      & \textbf{0.99956}       \\ \cline{3-9}
			&                     & phase 5( 9.459\%) & \textbf{1.00000}       & \textbf{1.00000}         & 0.98720        & \textbf{1.00000}         & 0.99919      & \textbf{0.99984}       \\ \cline{2-9}
			& \multicolumn{2}{c|}{$\beta$}             & $3.6\times10^4$  & $4.4\times10^4$    & $1.72\times10^5$   & $8.3\times10^4$    & $6.9\times10^4$     & $8.4\times10^4$   \\ 
			\hline
			\multirow{7}{*}{Phantom}   & \multirow{6}{*}{JS} & phase 1(58.177\%) & 0.99835       & \textbf{0.99995}         & 0.99040        & \textbf{0.99903}         & 0.99725      & \textbf{0.99976}       \\ \cline{3-9}
			&                     & phase 2(32.927\%) & 0.99773       & \textbf{0.99968}         & 0.97888        & \textbf{0.99519}         & 0.99412      & \textbf{0.99861}       \\ \cline{3-9}
			&                     & phase 3( 4.343\%) & 0.93008       & \textbf{1.00000}         & 0.99684        & \textbf{1.00000}         & \textbf{1.00000}      & \textbf{1.00000}       \\ \cline{3-9}
			&                     & phase 4( 4.335\%) & 0.96789       & \textbf{0.99789}         & 0.90039        & \textbf{0.96983}         & 0.95959      & \textbf{0.98951}       \\ \cline{3-9}
			&                     & phase 5( 0.139\%) & 0.55128       & \textbf{0.96703}         & 0.15024        & \textbf{0.60684}         & 0.37173      & \textbf{0.84694}       \\ \cline{3-9}
			&                     & phase 6( 0.079\%) & 0.00000       & \textbf{1.00000}         & 0.40909       & \textbf{0.88889}        & 0.00000      & \textbf{0.92593}       \\ \cline{2-9}
			& \multicolumn{2}{c|}{$\beta$}             & $5.8\times10^4$  & $2.3\times10^4$    & $1.45\times10^5$   & $6.4\times10^4$   & $6.2\times10^4$   & $3.5\times10^4$       \\ 
			\hline
		\end{tabular}
	}
	\caption{
		JS values of segmentation results for Geometry and Phantom corrupted by Gaussian noise with mean $0$ and variance $10^{-6}$. The percentages in the brackets are the ratios of the phase areas to the whole image in the ground truth.}
	\label{L2JSvalue}
\end{table}

%% file: section6.tex
In this paper, we studied a non-Lipschitz restoration model with anisotropic regularization and the $\ell_q,q\in[1,+\infty)$ fidelity.
We proved the lower bound theory for the case of $q=1$, and for a general $q\in[1,+\infty)$, the support inclusion property was  derived.
For solving such a nonconvex and non-Lipschitz model, we proposed the inexact iterative thresholding and support shrinking algorithm with proximal linearization, which is shown to globally converge to a stationary point of the objective function. The proof techniques used in such analysis can be applied to other image models with anisotropic regularizations.
Numerical experiment on image deconvolution and two stage image segmentation also illustrate the advantages of the proposed algorithm in applications.

%% file: section7.tex
\subsection{Subdifferential}\label{app:sub}
\begin{definition}~\cite[Definition 8.3]{rockafellar2009variational}
    Let $\omega:\RR^n \rightarrow \RR \cup \{+\infty\}$ be a proper lower semicontinuous function. The domain of $\omega$ is defined by
    $\text{dom} \omega := \{y\in \RR^n: \omega(y)<+\infty)\}$,
    \begin{enumerate}
        \item For each $y \in \text{dom } \omega$, the regular subdifferential of $\omega$ at $y$ is defined as:
               \[\widehat{\partial}\omega(y) :=\{y^\star \in \RR^n : \liminf_{\substack{z\to y\\ z\neq y}} \frac{\omega(z) - \omega(y) - \langle y^\star,z - y \rangle}{\lVert z - y\rVert} \geq 0 \}
              .\]
              If $y \notin \text{dom} \omega$, then $\hat{\partial}\omega(y) = \emptyset.$
        \item The limiting subdifferential of $\omega$ at $y$ is defined as:
              \[\partial \omega(y):=\{y^\star \in \RR^n : \exists y^k \rightarrow y, \omega(y^k)\rightarrow \omega(y), s^k\in \hat{\partial} \omega(y_k) \rightarrow y^\star \quad \text{as}\quad k\rightarrow +\infty\}.\]
        \item The horizon subdifferential of $\omega$ at $y$ is defined as:
        \begin{equation*}
        \begin{aligned}
        \partial^{\infty} \omega(y):= \{y^\star \in \RR^n : \exists \nu^k \rightarrow 0, \nu^k > 0, &\exists y^k \rightarrow y, \omega(y^k)\rightarrow \omega(y), \\
        &s^k\in \hat{\partial} \omega(y^k), \nu^k s^k \rightarrow y^\star \quad \text{as}\quad k\rightarrow +\infty\}.
        \end{aligned}
        \end{equation*}
    \end{enumerate}
\end{definition}
\begin{remark}
        A point $y$ is said to be a stationary point of $\omega$, if  $0\in \partial \omega(y)$.
\end{remark}

\subsection{Kurdyka-{\L}ojasiewicz function}
The definition for a proper lower semicontinuous function $f$ to have the Kurdyka-{\L}ojasiewicz (KL) property at $\bar{x}\in \text{dom} \partial f$ can be found in \cite[Definition 3.1]{Attouch2010Proximal}.
A proper lower semicontinuous function $f$ satisfying the KL property at all points in $\text{dom}\partial f$ is called a \textit{KL function}.
A large class of KL functions widely used in applications are given by functions definable in an o-minimal structure introduced in \cite{Lou1996Geometric}.
\begin{definition}~\cite[Definition 4.1]{Attouch2010Proximal}\label{def:o-minimal}
Let $\mathcal{O}=\{\mathcal{O}_{n}\}_{n\in N}$  such that each $\mathcal{O}_{n}$ is a collection of subsets of $\mathbb{R}^{n}$. The family $\mathcal{O}$ is an
o-minimal structure over $\mathbb{R}$, if it satisfies the following axioms:
\begin{itemize}
    \item[1.] Each $\mathcal{O}_{n}$ is a boolean algebra. Namely $\emptyset \in \mathcal{O}_n$ and for each $A,B\in \mathcal{O}_{n}, A\cup B, A\cap B,$ and $\mathbb{R}^{n}\setminus A$ belong to $\mathcal{O}_{n}$.
    \item[2.] For all $A\in \mathcal{O}_{n}, A\times R$ and $R\times A$ belong to $\mathcal{O}_{n+1}$.
    \item[3.] For all $A\in \mathcal{O}_{n+1}$, $\prod(A):=\{(x_1,...,x_n)\in \mathbb{R}^{n}\;|\; (x_1,...,x_n,x_{n+1})\in A\}$ belongs to $\mathcal{O}_{n}$.
    \item[4.] For all $i\neq j$ in $\{1,2,...,n\}$, $\{(x_1,...,x_n)\in \mathbb{R}^{n}\;|\; x_i=x_j\}$ belong to $\mathcal{O}_{n}$.
    \item[5.] The set $\{(x_1,x_2)\in \mathbb{R}^2\;|\;x_1<x_2\}$ belong to $\mathcal{O}_{2}$.
    \item[6.] The elements of $\mathcal{O}_{1}$ are exactly finite unions of intervals.
\end{itemize}
\end{definition}
We say that a set $A \subseteq \RR^n$ belongs to $\mathcal{O}$ if $A\in \mathcal{O}_{n}$. A map $\Psi: A \rightarrow \RR^m$ with $A\subseteq \RR^n$ is said to belong to $\mathcal{O}$ if its graph $\{(x,\Psi(x))|x\in\text{dom}\Psi\} \subseteq \RR^{n+m}$ belongs to $\mathcal{O}$. We say that sets and maps are definable in $\mathcal{O}$ if they belong to $\mathcal{O}$. 
Definable functions are defined like definable maps.
By \cite{Attouch2010Proximal}, the o-minimal structure has the properties that 1) the composition of definable functions is definable; 2) the finite sum of definable functions is definable.

A class of o-minimal structure is the log-exp structure given in~\cite[Example 2.5]{Lou1996Geometric}, by which the following functions are definable:
\begin{itemize}
  \item[1.] semi-algebraic functions~\cite[Definition 5]{bolte2014proximal}, including real polynomial functions.
  \item[2.]  $x^{r}: \RR\rightarrow \RR$ with $r\in\mathbb{R}$, which is given by
\[
	a \mapsto\begin{cases}
	a^r, & a>0\\
	0, &a\leq 0.
	\end{cases}
\]
  \item[3.] the exponential function: $\mathbb{R}\rightarrow \mathbb{R}$ defined by $x \mapsto e^{x}$ and the logarithm function: $(0, \infty) \rightarrow \mathbb{R}$ defined by $x \mapsto \text{log}(x)$.
\end{itemize}
It has been shown that any proper lower semicontinuous function that is definable in an o-minimal structure is a KL function (\cite{bolte2007clarke} and~\cite[Theorem 4.1]{Attouch2010Proximal}).
Then by the aforementioned properties and examples of definable functions, the objective functions $\mathcal{F}(x)$ in the examples of this paper are KL functions.

%% file: anisotropic-nonlipschitz.bbl
\begin{thebibliography}{10}
\expandafter\ifx\csname url\endcsname\relax
  \def\url#1{\texttt{#1}}\fi
\expandafter\ifx\csname urlprefix\endcsname\relax\def\urlprefix{URL }\fi
\expandafter\ifx\csname href\endcsname\relax
  \def\href#1#2{#2} \def\path#1{#1}\fi

\bibitem{Bouman1993A}
C.~Bouman, K.~Sauer, A generalized {Gaussian} image model for edge-preserving
  {{MAP}} estimation, IEEE Trans. Image Process. 2~(3) (1993) 296--310.

\bibitem{Chen2001Atomic}
S.~S. Chen, D.~M.~A. Saunders, Atomic decomposition by basis pursuit, SIAM Rev.
  43~(1) (2001) 129--159.

\bibitem{candes2006stable}
E.~J. Candes, J.~Romberg, T.~Tao, Stable signal recovery from incomplete and
  inaccurate measurements, Comm. Pure Appl. Math. 59~(8) (2006) 1207--1223.

\bibitem{nikolova2004a}
M.~Nikolova, A variational approach to remove outliers and impulse noise, J.
  Math. Imaging Vision 20~(1) (2004) 99--120.

\bibitem{granai2005}
L.~Granai, P.~Vandergheynst, Sparse approximation by linear programming using
  an ${L}_1$ data-fidelity term, in: Proceedings of Workshop on Signal
  Processing with Adaptative Sparse Structured Representations, no. CONF, 2005.

\bibitem{rodriguez2008}
P.~Rodriguez, B.~Wohlberg, An efficient algorithm for sparse representations
  with $\ell^p$ data fidelity term, in: Proceedings of 4th IEEE Andean
  Technical Conference (ANDESCON), 2008.

\bibitem{Fuchs2009Fast}
J.-J. Fuchs, Fast implementation of a $\ell_1$-$\ell_1$ regularized sparse
  representations algorithm, in: Proceedings of the IEEE International
  Conference on Acoustics, Speech and Signal Processing, 2009, pp. 3329--3332.

\bibitem{nikolova2005analysis}
M.~Nikolova, Analysis of the recovery of edges in images and signals by
  minimizing nonconvex regularized least-squares, Multiscale Model. Simul.
  4~(3) (2005) 960--991.

\bibitem{chartrand2007exact}
R.~Chartrand, Exact reconstruction of sparse signals via nonconvex
  minimization, IEEE Signal Process. Lett. 14~(10) (2007) 707--710.

\bibitem{Foucart2009Sparsest}
S.~Foucart, M.-J. Lai, Sparsest solutions of underdetermined linear systems via
  $\ell_q$-minimization for $0<q \leq 1$, Appl. Comput. Harmon. Anal. 26~(3)
  (2009) 395--407.

\bibitem{nikolova2008Efficient}
M.~Nikolova, M.~K. Ng, S.~Zhang, W.~K. Ching, Efficient reconstruction of
  piecewise constant images using nonsmooth nonconvex minimization, SIAM J.
  Imaging Sci. 1~(1) (2008) 2--25.

\bibitem{nikolova2010fast}
M.~Nikolova, M.~K. Ng, C.~P. Tam, Fast nonconvex nonsmooth minimization methods
  for image restoration and reconstruction, IEEE Trans. Image Process. 19~(12)
  (2010) 3073--88.

\bibitem{chen2012non}
X.~Chen, M.~K. Ng, C.~Zhang, Non-lipschitz $\ell_{p}$-regularization and box
  constrained model for image restoration, IEEE Trans. Image Process. 21~(12)
  (2012) 4709--4721.

\bibitem{2018fengOn}
F.~Xue, W.~Chunlin, Z.~Chao, On the local and global minimizers of $\ell_0$
  gradient regularized model with box constraints for image restoration,
  Inverse Problems 34~(9) (2018) 095007.

\bibitem{zeng2019nonlip}
C.~Zeng, C.~Wu, R.~Jia, Non-lipschitz models for image restoration with impulse
  noise removal, SIAM J. Imaging Sci. 12~(1) (2019) 420--458.

\bibitem{liu2019a}
Z.~Liu, C.~Wu, Y.~Zhao, A new globally convergent algorithm for non-lipschitz
  $\ell_p$-$\ell_q$ minimization, Adv. Comput. Math. 45~(3) (2019) 1369--1399.

\bibitem{2019zengOn}
Z.~Chao, W.~Chunlin, On the discontinuity of images recovered by noncovex
  nonsmooth regularized isotropic models with box constraints, Adv. Comput.
  Math 45 (2019) 589--610.

\bibitem{zeng2018edge}
C.~Zeng, C.~Wu, On the edge recovery property of noncovex nonsmooth
  regularization in image restoration, SIAM J. Numer. Anal. 56~(2) (2018)
  1168--1182.

\bibitem{glowinski1989augmented}
R.~Glowinski, P.~Le~Tallec, Augmented {Lagrangian} and operator-splitting
  methods in nonlinear mechanics, Vol.~9, SIAM, 1989.

\bibitem{goldstein2009the}
T.~Goldstein, S.~Osher, The split {Bregman} method for {L1}-regularized
  problems, SIAM J. Imaging Sci. 2~(2) (2009) 323--343.

\bibitem{li2015global}
G.~Li, T.~Pong, Global convergence of splitting methods for nonconvex composite
  op timization, SIAM J. Optim. 25~(4) (2015) 2434--2460.

\bibitem{hong2016convergence}
M.~Hong, Z.~Luo, M.~Razaviyayn, Convergence analysis of alternating direction
  method of multipliers for a family of nonconvex problems, SIAM J. Optim.
  26~(1) (2016) 337--364.

\bibitem{wang2019global}
Y.~Wang, W.~Yin, J.~Zeng, Global convergence of {ADMM} in nonconvex nonsmooth
  optimization, J. Sci. Comput. 78~(1) (2019) 29--63.

\bibitem{chen2012smoothing}
X.~Chen, Smoothing methods for nonsmooth, nonconvex minimization, Math.
  Program. 134~(1) (2012) 71--99.

\bibitem{bian2015linearly}
W.~Bian, X.~Chen, Linearly constrained non-lipschitz optimization for image
  restoration, SIAM J. Imaging Sci. 8~(4) (2015) 2294--2322.

\bibitem{chen2013optimality}
X.~Chen, L.~Niu, Y.~Yuan, Optimality conditions and a smoothing trust region
  {Newton} method for non-{Lipschitz} optimization, SIAM J. Optim. 23~(3)
  (2013) 1528--1552.

\bibitem{hintermuller2014a}
M.~Hintermuller, T.~Wu, A superlinearly convergent {R-regularized Newton}
  scheme for variational models with concave sparsity-promoting priors, Comput.
  Optim. Appl. 57~(1) (2014) 1--25.

\bibitem{chan2014half}
R.~H. Chan, H.-X. Liang, Half-quadratic algorithm for $\ell_p$-$\ell_q$
  problems with applications to {TV}-$\ell_1$ image restoration and compressed
  sensing, in: Efficient Algorithms for Global Optimization Methods in Computer
  Vision. Lecture Notes in Computer Science, Vol. 8293, 2014, pp. 78--103.

\bibitem{rodriguez2009efficient}
P.~Rodriguez, B.~Wohlberg, Efficient minimization method for a generalized
  total variation functional, IEEE Trans. Image Process. 18~(2) (2009)
  322--332.

\bibitem{lanza2015krylov}
A.~Lanza, S.~Morigi, L.~Reichel, F.~Sgallari, A generalized {Krylov} subspace
  method for $\ell_p$-$\ell_q$ minimization, SIAM J. Sci. Comput. 37~(5) (2015)
  S30--S50.

\bibitem{xue2019efficient}
Y.~Xue, Y.~Feng, C.~Wu, An efficient and globally convergent algorithm for
  $\ell_{p,q}$-$\ell_{r}$ model in group sparse optimization, arXiv preprint
  arXiv:1904.01887.

\bibitem{liunew}
Z.~Liu, C.~Wu, Y.~Zhao, A new globally convergent algorithm for non-lipschitz
  $\ell_p$-$\ell_q$ minimization, Adv. Comput. Math.  1--31.

\bibitem{feng2020the}
X.~Feng, S.~Yan, C.~Wu, The $\ell_{2,q}$ regularized group sparse optimization:
  {Lower} bound theory, recovery bound and algorithms, Appl. Comput. Harmon.
  Anal. 49~(2) (2020) 381--414.

\bibitem{zeng2019iterative}
C.~Zeng, R.~Jia, C.~Wu, An iterative support shrinking algorithm for
  non-lipschitz optimization in image restoration, J. Math. Imaging Vis. 61~(1)
  (2019) 122--139.

\bibitem{zheng2020a}
Z.~Zheng, M.~Ng, C.~Wu, A globally convergent algorithm for a class of gradient
  compounded non-{Lipschitz} models applied to non-additive noise removal,
  Inverse Problems 36~(12) (2020) 125017.

\bibitem{Gao2021a}
Y.~Gao, C.~Wu, A general non-lipschitz joint regularized model for
  multi-channel/modality image reconstruction, CSIAM Transactions on Applied
  Mathematics (2021) Accepted.

\bibitem{Candes2008IRL1}
E.~J. Cand\`es, M.~B. Wakin, S.~P. Boyd, Enhancing sparsity by reweighted
  $\ell_1$ minimization, J. Fourier Anal. Appl. 14~(5-6) (2008) 877--905.

\bibitem{Lai2009IRL1}
S.~Foucart, M.~J. Lai, Sparsest solutions of underdetermined linear systems via
  $\ell_q$-minimization for $0 < q <= 1$, Appl. Comput. Harmon. Anal. 26~(3)
  (2009) 395--407.

\bibitem{Daubechies2010IRLS}
I.~Daubechies, R.~Devore, M.~Fornasier, C.~S. Gunturk, Iteratively reweighted
  least squares minimization for sparse recovery, Comm. Pure Appl. Math. 63~(1)
  (2010) 1--38.

\bibitem{Lai2013IRLS}
M.~J. Lai, Y.~Y. Xu, W.~T. Yin, Improved iteratively reweighted least squares
  for unconstrained smoothed $\ell_q$ minimization, SIAM J. Numer. Anal. 51~(2)
  (2013) 927--957.

\bibitem{Figueiredo2007Majorization}
M.~A.~T. Figueiredo, J.~M. Bioucas-Dias, R.~D. Nowak, Majorization-minimization
  algorithms for wavelet-based image restoration, IEEE Trans. Image Process.
  16~(12) (2007) 2980--2991.

\bibitem{rockafellar2009variational}
R.~T. Rockafellar, R.~J.-B. Wets, Variational analysis, Vol. 317, Springer,
  2009.

\bibitem{Wang2008}
Y.~Wang, J.~Yang, W.~Yin, Y.~Zhang, A new alternating minimization algorithm
  for total variation image reconstruction, SIAM Journal on Imaging Sciences
  1~(3) (2008) 248--272.

\bibitem{Wu2010Augmented}
C.~Wu, X.-C. Tai, Augmented {Lagrangian} method, dual methods, and split
  {Bregman} iteration for {ROF}, vectorial {TV}, and high order models, SIAM J.
  Imaging Sci. 3~(3) (2010) 300--339.

\bibitem{lou2015weighted}
Y.~Lou, T.~Zeng, S.~Osher, J.~Xin, A weighted difference of anisotropic and
  isotropic total variation model for image processing, SIAM J. Imaging Sci.
  8~(3) (2015) 1798--1823.

\bibitem{Attouch2013Convergence}
H.~Attouch, J.~Bolte, B.~F. Svaiter, Convergence of descent methods for
  semi-algebraic and tame problems: Proximal algorithms, forward-backward
  splitting, and regularized {{Gauss}}-{{Seidel}} methods, Math. Program.
  137~(1-2) (2013) 91--129.

\bibitem{zhang2017nonconvexTV}
X.~Zhang, M.~Bai, K.~N. Michael, Nonconvex-{TV} based image restoration with
  impulse noise removal, SIAM J. Imaging Sci. 10~(3) (2017) 1627--1667.

\bibitem{Lojasiewicz1963Une}
S.~{\L}ojasiewicz, Une propri\'et\'e topologique des sous-ensembles analytiques
  r\'eels, in: Les \'Equations aux {D}\'eriv\'ees {P}artielles ({P}aris, 1962),
  1963, pp. 87--89.

\bibitem{Kurdyka1998gradients}
K.~Kurdyka, On gradients of functions definable in o-minimal structures, Ann.
  Inst. Fourier (Grenoble) 48~(3) (1998) 769--783.

\bibitem{absil2005convergence}
P.~Absil, R.~Mahony, B.~Andrews, Convergence of the iterates of descent methods
  for analytic cost functions, SIAM J. Optim. 16~(2) (2005) 531--547.

\bibitem{bolte2007the}
J.~Bolte, A.~Daniilidis, A.~Lewis, The {{\L}ojasiewicz} inequality for
  nonsmooth subanalytic functions with applications to subgradient dynamical
  systems, SIAM J. Optim. 17~(4) (2007) 1205--1223.

\bibitem{Attouch2009convergence}
H.~Attouch, J.~Bolte, On the convergence of the proximal algorithm for
  nonsmooth functions involving analytic features, Math. Program. 116 (2009)
  5--16.

\bibitem{Attouch2010Proximal}
H.~Attouch, J.~Bolte, P.~Redont, A.~Soubeyran, Proximal alternating
  minimization and projection methods for nonconvex problems: {An} approach
  based on the {Kurdyka-{\L}ojasiewicz} inequality, Math. Oper. Res. 35~(2)
  (2010) 438--457.

\bibitem{bolte2007clarke}
J.~Bolte, A.~Daniilidis, A.~Lewis, M.~Shiota, Clarke subgradients of
  stratifiable functions, SIAM J. Optim. 18~(2) (2007) 556--572.

\bibitem{Lou1996Geometric}
L.~van~den Dries, C.~Miller, Geometric categories and o-minimal structures,
  Duke Math. J. 84~(2) (1996) 497--540.

\bibitem{boyd2011distributed}
S.~{Boyd}, N.~{Parikh}, E.~{Chu}, B.~{Peleato}, J.~{Eckstein}, Distributed
  Optimization and Statistical Learning via the Alternating Direction Method of
  Multipliers, 2011.

\bibitem{he2012on}
B.~He, X.~Yuan, On the {O(1/n)} convergence rate of the {D}ouglas-{R}achford
  alternating direction method, SIAM J. Numer. Anal. 50~(2) (2012) 700--709.

\bibitem{glowinski2016splitting}
R.~Glowinski, S.~Osher, W.~Yin, Splitting Methods in Communication, Imaging,
  Science, and Engineering, Springer International Publishing, 2016.

\bibitem{donoho1994ideal}
D.~L. Donoho, I.~M. Johnstone, {Ideal spatial adaptation by wavelet shrinkage},
  Biometrika 81~(3) (1994) 425--455.

\bibitem{Donoho1995DenoisingBS}
D.~Donoho, De-noising by soft-thresholding, IEEE Trans. Inform. Theory 41
  (1995) 613--627.

\bibitem{cai2013segmentation}
X.~Cai, R.~Chan, T.~Zeng, A two-stage image segmentation method using a convex
  variant of the mumford-shah model and thresholding, SIAM J. Imaging Sci.
  6~(1) (2013) 368--390.

\bibitem{jaccard1912the}
P.~Jaccard, The distribution of the flora in the alpine zone, New Phytol.
  11~(2) (1912) 37--50.

\bibitem{bolte2014proximal}
J.~B. Bolte, S.~Sabach, M.~Teboulle, Proximal alternating linearized
  minimization for nonconvex and nonsmooth problems, Math. Program. 146~(1-2)
  (2014) 459--494.

\end{thebibliography}
